\documentclass[11pt]{amsart}

\usepackage{graphicx}
\usepackage{amsmath}
\usepackage{amssymb}
\usepackage{amsopn}
\usepackage{mathrsfs}
\usepackage{enumerate}
\usepackage[osf]{newtxtext}
\usepackage[hidelinks]{hyperref}

\DeclareMathOperator*{\Div}{div}
\DeclareMathOperator*{\psh}{psh}

\DeclareMathOperator{\interior}{int}

\DeclareMathOperator*{\ddc}{dd^c}
\DeclareMathOperator*{\grad}{grad}
\DeclareMathOperator{\Vol}{Vol}
\DeclareMathOperator{\Span}{span}
\DeclareMathOperator{\Capa}{cap}
\DeclareMathOperator{\ric}{Ric}
\DeclareMathOperator{\Sqrt}{Sqrt}
\newcommand{\Ric}{\boldsymbol{\ric}}

\renewcommand{\wp}{\mathscr{P}}

\renewcommand{\epsilon}{\varepsilon}

\newcommand{\ddcn}[1]{\left(\ddc #1\right)^n}

\newcommand{\C}{\mathbb{C}}
\newcommand{\R}{\mathbb{R}}
\newcommand{\N}{\mathbb{N}}

\renewcommand{\phi}{\varphi}
\newcommand{\z}{\zeta}

\newtheorem{theorem}{Theorem}
\newtheorem{conjecture}{Conjecture}
\newtheorem{cor}{Corollary}[section]
\newtheorem{proposition}{Proposition}[section]
\newtheorem{lemma}{Lemma}[section]

\newtheorem{definition}[theorem]{Definition}
\newtheorem{remark}{Remark}
\newtheorem{question}{Question}
\newcounter{theexample} 
    
\begin{document}

% ARTICLE INFORMATIONS
\author{Federico Piazzon}
\address{room 712 Department of Mathematics, Universit\'a di Padova, Italy. 
Phone +39 0498271260}
\email{\underline{fpiazzon@math.unipd.it}} 
\urladdr{http://www.math.unipd.it/~fpiazzon/   (work in progress)}

%\keywords{Bernstein Markov Property, rational approximation, logarithmic capacity, convergence of Greeen functions}
%\thanks{Dep. of Mathematics University of Padua, Italy. Supported by Doctoral School on Mathematical Sciences.}
\date{\today}
\title[]{Laplace Beltrami operator in the Baran metric and pluripotential equilibrium measure: the ball, the simplex and the sphere.}

\maketitle

\begin{abstract}
The Baran metric $\delta_E$ is a Finsler metric on the interior of $E\subset \R^n$ arising from Pluripotential Theory. We consider the few instances, namely $E$ being the ball, the simplex, or the sphere, where $\delta_E$ is known to be Riemaniann and we prove that the eigenfunctions of the associated Laplace Beltrami operator  (with no boundary conditions) are the orthogonal polynomials with respect to the pluripotential equilibrium measure $\mu_E$ of $E.$ We conjecture that this may hold in a wider generality. 

The considered differential operators have been already introduced in the framework of orthogonal polynomials and studied in connection with certain symmetry groups. In this work instead we highlight the relationships between orthogonal polynomials with respect to $\mu_E$ and the Riemaniann structure naturally arising from Pluripotential Theory.
\end{abstract}
\small
\tableofcontents
\normalsize
\section{Introduction}
\subsection{Potential Theory and polynomials}
The study of Approximation Theory in the complex plane and on the real line (by polynomials and rational functions) is deeply related to the Logarithmic Potential Theory (i.e., the study of subharmonic functions and Laplace operator); this is a classical and well established topic whose study goes back to Fekete, Leja, Szeg\"o, Walsh and many others. 

These relations between Logarithmic Potential Theory and Approximation Theory spread among Markov, Bernstein and Nikolski type polynomial inequalities, asymptotic of optimal polynomial interpolation arrays and Fekete points, \emph{overconvergence phenomena} (i.e. uniformly convergent sequence of polynomials defining a holomorphic function in a larger open set) and its quantitative version, the  Bernstein Walsh  Theorem , asymptotic of orthogonal polynomials, random polynomials and random matrices. Moreover, most of such relations extend to the more general case of weighted polynomials and Logarithmic Potential Theory in presence of an external field. We refer to \cite{SaTo97,Sa10,StaTo92,WAL,Rans} and references therein for extensive treatments of this subjects.  

More recently a non linear potential theory in multi dimensional complex spaces has been introduced and many analogies with the linear case have been shown, provided a suitable ''translation'' of the quantities that come into the play. \emph{Pluripotential Theory} (see for instance \cite{Kli,Klo05}) is the study of plurisubharmonic functions (i.e., functions which are subharmonic along each complex line) and the complex Monge Ampere operator; \cite{BeTa82}.

Though the lack of linearity makes this new theory much more difficult and requires to work with different tools, many connections with polynomial approximation has been extended to this multi dimensional framework; see \cite{LevSurRocky92,LevSur}. Indeed, polynomial inequalities in $\C^n$  are usually obtained by means of Pluripotential Theory, see for instance \cite{BaBi14,Ba92}, the Bernstein Walsh Theorem has been extended by Siciak to $\C^n$ \cite{Si81} and more general complex spaces by  Zeriahi \cite{Ze91}. In his seminal work \cite{Za74I,Za74II,Za75}, Zaharjuta extended  the equivalence between (a suitably re-defined version of) the Chebyshev Constant (i.e., the asymptotic of the uniform norms of monic polynomials) and the Transfinite Diameter (i.e., the asymptotic of the maximum of the Vandermonde determinant). Very recently, Berman Boucksom and Nystrom \cite{BeBo10,BeBoNy11} showed that Fekete points converge weak$^*$ to the \emph{pluripotential equilibrium measure} of the considered set in $\C^n$ and in much more general settings, this is a deep extension of the one dimensional case which can be obtained only by means of the weighted theory. The work of Berman and Boucksom stimulated different lines of research as $L^2$ theory and general orthogonal polynomials \cite{B97}, the study of multi-variate random polynomials \cite{ZeZe10,BlLe15,PrYeeAa15}, and the theory of sampling and interpolation arrays \cite{MaOr10,BeOr15,MaOr15}. More importantly from our point of view, the widely used heuristic that \emph{the "best" measure for producing uniform polynomial approximations by $L^2$ projection is the equilibrium measure} has been fully motivated and theoretically explained in \cite{BeBoNy11} also in its multivariate setting.    
\vskip 0.5cm
The present work concerns, on one hand, to (partially) extend to the $\C^n$ case another \emph{connection between polynomials and Potential Theory}, on the other hand, to highlight how \emph{the polynomial $L^2$ approximation with respect to the equilibrium measure may be regarded as Fourier Analysis on a suitable Riemaniann manifold.}  These ideas rest upon the relation between the Laplace Beltrami operator relative to the Baran metric and the orthogonal polynomials with respect to the pluripotential equilibrium measure.

We would like to introduce such relations starting by some examples that treat the instances of the interval $[-1,1]$ and the unit sphere.

\subsection{Two motivational examples}
\subsubsection{Chebyshev polynomials}
Chebyshev polynomials $T_n(x):=\arccos(n\cos x)$ are the orthogonal polynomials with respect to $\frac 1 {\pi\sqrt{1-x^2}}dx$, the equilibrium measure  of the interval $[-1,1]$ as a subset of $\C,$ i.e., the unique minimizer of the logarithmic potential $-\int \log|z-w|d\mu(z)d\mu(w)$ among all Borel probability measures $\mu$ on the interval $[-1,1]$. Another classical characterization of Chebyshev polynomials is given by the eigen-functions of the Sturm-Liouville eigenvalue problem 
\begin{equation}
\label{STeigenproblem}
\begin{cases}\mathscr S[\phi](x):=(1-x^2)\phi''(x)-x\phi'(x)=-\lambda\phi(x),& x\in ]-1,1[\\
\phi'(x_0)=0,& x_0\in\{-1,1\}
\end{cases}.
\end{equation}
The set of eigenvalues turns out to be $\{n^2: n\in \N\}$ and $\mathscr S[T_n]=n^2 T_n.$ 

Instead, we re-write such an eigenvalue problem as
\begin{equation}\label{SLE}
\frac 1{\frac 1{\sqrt{1-x^2}}}\frac{d}{dx}\left(\frac 1{\sqrt{1-x^2}} (1-x^2)   \phi'(x)\right)=-n^2\phi(x),\;\;x\in ]-1,1[.
\end{equation}   
This apparently useless manipulation actually enlightens another property of Chebyshev polynomials. To explain this property, we first recall that the Laplace Beltrami operator relative to a metric $g$ can be written in local coordinates as
\begin{equation}\label{laplacebeltramidef}
\Delta_{LB}f=\frac 1{\sqrt {\det g}}\sum_{i=1}^n\partial_{x_i}\left(\sqrt{ \det g} \sum_j g^{i,j}\partial _{x_j}f  \right),
\end{equation} 
where $g^{i,j}$ are the components of the inverse of the matrix representing $g.$

Let us endow $]-1,1[$ with the Riemaniann metric $g(x):=\frac 1{1-x^2},$ we canonically obtain the Riemannian distance $d(x_0,x_1)=\int_{x_0}^{x_1}\frac 1 {\sqrt{1-x^2}}dx.$ Note that, up to a re-normalization, the resulting volume form is precisely the equilibrium measure of $[-1,1].$  If we plug $g(x):=\frac 1{1-x^2}$ in the expression \eqref{laplacebeltramidef} of the Laplace Beltrami operator, we obtain precisely the left hand side of \eqref{SLE}. In other words, we observe that: 
\begin{itemize}
\item\emph{Chebyshev polynomials are eigenfunctions of the Laplace Beltrami operator with respect to the equilibrium measure of the interval.}
\end{itemize}  
It  is relevant to notice that the density of the equilibrium measure on $[-1,1]$ at $x$ is obtained as the normal (i.e., purely complex) derivative of the Green function of $\C\setminus [-1,1]$ with pole at infinity; see \cite[Ch II.1]{SaTo97}. This operation has a multidimensional counterpart (see \cite{Ba92}) that, under some assumptions, leads to the so called \emph{Baran metric} (\cite{BlBoLe12,BoLeVi12,Ba95}), see equation \eqref{baranmetricdef} below.

\subsubsection{Spherical harmonics}
We mention another relevant example of this relation between eigenfunctions of the Laplace Beltrami operator with respect to the metric defined by (pluri-)potential theory and the (pluripotential) equilibrium measure. In contrast with the case of Chebyshev polynomials, now we work in a multi dimensional setting and the flat euclidean space $\C^n$ is replaced by a complex manifold. A more detailed account of this example requires some preliminary notions in addition to the ones of Subsection \ref{subsectionPluripot}, thus we decided to present the explicit computations in Appendix \ref{appendix}, together with the needed recalls from pluripotential theory on algebraic varieties. At this stage we only sketch the results to underline the analogy with the case of Chebyshev polynomials. 

Let us consider the unit sphere $\mathbb S^{n-1}\subset \R^n$ endowed it with the \emph{round metric} $g$ induced by the flat metric on $\R^n$ and denote by $\Delta$ the Laplace Beltrami operator on $\mathbb S^{n-1}.$  It is well known that \emph{spherical harmonics} are a dense orthogonal system of $L^2(\mathbb S^{n-1})$ which consists of polynomials that are eigenfunctions of $\Delta.$ 

Let us look at $\mathbb S^{n-1}$ as a compact subset of the \emph{complexified sphere} $\mathcal S^{n-1}:=\{z\in \C^n:\sum z_i^2 =1\}.$ By a fundamental result due to Sadullaev \cite{Sa82}, since $\mathcal S^{n-1}$ is a irreducible algebraic variety, one can relate ( see Appendix \ref{appendix}) the traces of polynomials on $\mathcal S^{n-1}$  to the \emph{pluripotential theory on the complex manifold} of $\mathcal S^{n-1}.$ On the other hand, due to Lemma \ref{differentiabilitylemma} below,  we can define a smooth Riemaniann metric $g_{\mathbb S^{n-1}}$ on $\mathbb S^{n-1}$ suitably modifying the construction (see eq. \eqref{baranmetricdef}) of the Baran metric of convex real bodies. In particular such a definition is given by the generalization of the case of the real interval $[-1,1]$. Indeed, it turns out that $g_{\mathbb S^{n-1}}=g$ and its volume form is, up to a constant scaling factor, the pluripotential equilibrium measure (see equation \eqref{eqmsph} below) of $\mathbb S^{n-1},$ as compact subset of $\mathcal S^{n-1}.$ In other words 
\begin{itemize}
\item\emph{the eigenfunctions of the Laplace Beltrami operator of $(\mathbb S^{n-1},g)$ are the orthogonal polynomials with respect to the pluripotential equilibrium measure of $\mathbb S^{n-1},$ seen as compact subset of $\mathcal S^{n-1};$} see Corollary \ref{corollarysphere}.
\end{itemize}

\subsection{Our results and conjecture}
The aim of the present paper is to present a conjecture on the extension to the $\C^n$ case of the relation between potential theory and certain Riemaniann structure that holds in the examples above. We support it by full proofs of all the few known instances fulfilling the required hypothesis, see Theorems \ref{ThBall} and \ref{ThSimplex} below.

\begin{conjecture}\label{Conjecture}
Let $\mathcal C$ denote either $\C^n$ or any irreducible algebraic sub-variety of it. Let $E\subset \mathcal C$ be a fat\footnote{This should be intended as the closure in $\mathcal C$ of the interior in $ \R^n\cap \mathcal C$ of $E$ equals to $E$ itself.} real compact set. Assume that the \emph{Baran metric} $\delta_E$ of $E$ is a Riemannian metric on $\interior_{\R^n\cap \mathcal C}E,$ then the orthonormal polynomials with respect to the pluripotential equilibrium measure $\mu_{E,\mathcal C}$ of $E$ in $\mathcal C$ are eigenfunctions of the Laplace Beltrami operator relative to the metric $\delta_E.$ 
\end{conjecture}
\begin{remark}
We stress that the orthogonal bases used in our proofs as well as most of their properties are already known in the framework of orthogonal polynomials (see \cite{DuXu01,DuXu14} and references therein). Moreover, our differential operators (i.e., Laplace Beltrami operators with respect to the metrics arising from Pluripotential Theory) turn out to be already studied in relation with certain symmetry groups \cite[Ch. 8]{DuXu14}, but they have not been related to any potential theoretic aspects before. More precisely, the Laplace Beltrami operator on the ball endowed with its Baran metric turns out to be the operator $\mathscr D_\mu$ in \cite[pg. 142]{DuXu01} with the parameter choice $\mu=0$. Instead, in the simplex case, $\Delta$ is precisely the operator defined in \cite[eq. 5.3.4]{DuXu01} (see equation \ref{differentialproperty} and Theorem \ref{simplexeigenfunctions} below) if we set (in the authors notation) $\kappa=(0,\dots,0)\in \R^{n+1}.$  

\emph{Our goal is precisely to relate such families of functions and their properties to the Riemaniann structure that comes from Pluripotential Theory. }  
\end{remark}

\begin{theorem}[Laplace Beltrami on the Baran Ball]\label{ThBall}
Let us denote by $\Delta$ the Laplace Beltrami operator of the Riemannian manifold $(B^n,g_{B^n})$ acting on 
$$\mathscr C^2_b(B^n):=\left\{ u\in \mathscr C^2(B^n):\max_{|\alpha|\leq 2}\sup_{x\in B^n}|\partial^\alpha u(x)|<\infty\right\},$$
where $B^n:=\{x\in \R^n: |x|< 1\}$ and $g_{B^n}(x)$ is represented by the matrix \small
 $$ G_{B^n}(x):=\left[\begin{array}{cccccc}
1+\frac{x_1^2}{1-\sum_{i=1}^nx_i^2}& \frac{x_1x_2}{1-\sum_{i=1}^nx_i^2}     & \frac{x_1x_3}{1-\sum_{i=1}^nx_i^2}&\dots& \dots &\frac{x_1x_n}{1-\sum_{i=1}^nx_i^2}\\
\frac{x_1x_2}{1-\sum_{i=1}^nx_i^2} & 1+\frac{x_2^2}{1-\sum_{i=1}^nx_i^2}& \frac{x_2x_3}{1-\sum_{i=1}^nx_i^2}&\dots& \dots &\frac{x_2x_n}{1-\sum_{i=1}^nx_i^2}\\
\vdots  & \vdots    & \vdots&\vdots&\vdots&\vdots\\
\frac{x_nx_1}{1-\sum_{i=1}^nx_i^2}& \frac{x_nx_2}{1-\sum_{i=1}^nx_i^2}     & \dots&\dots& \dots &1+\frac{x_n^2}{1-\sum_{i=1}^nx_i^2}
\end{array}\right].$$ \normalsize
The operator $\Delta$ is symmetric and unbounded, it 
 has discrete spectrum 
 $$\sigma(\Delta)=\{\lambda_s:=s(s+n-1): s\in \N\}$$
  and the eigen-space of $\lambda_s$ is $\Span\{\phi_\alpha,|\alpha|=s\},$ where $\phi_\alpha$ (see Proposition \ref{orthogonalsystemball})
are orthonormal polynomials with respect to the pluripotential equilibrium measure
$$\mu_{B^n}:=\frac 1{\sqrt{1-|x|^2}}\chi_{B^n} \Vol_{\R^n}=\Vol_{g_{B^n}}.$$

Moreover, $\Delta$ can be closed to a self-adjoint operator $\mathcal D(\Delta)\rightarrow L^2(B^n,g_{B^n})$ (having the same spectrum), where
$$\mathcal D(\Delta):=\left\{u\in L^2(B^n,g_{B^n}): \sum_{s=0}^\infty\lambda_s^{2}\sum_{|\alpha|=s} |\hat u_{\alpha}|^2<\infty\right\}\subset H^1(B^n,g_{B^n})$$
and $\hat u_\alpha$ is the Fourier coefficient $\int_{B^n} u \frac{\phi_\alpha}{\|\phi_\alpha\|^2_{L^2(\mu_{B^n})}} d\mu_{B^n}.$  

The operator $\Delta^{1/2}$ has domain
$$\mathcal D(\Delta^{1/2}):=\left\{u\in L^2(B^n,g_{B^n}): \sum_{s=0}^\infty\lambda_s\sum_{|\alpha|=s} |\hat u_{\alpha}|^2<\infty\right\}= H^1(B^n,g_{B^n})$$
\end{theorem}
For a precise definition of the Sobolev space $H^1(B^n,g_{B^n})$ see Subsection \ref{SubsectionSobolev} below.

%\vskip 1cm
\begin{theorem}[Laplace Beltrami on the Baran Simplex]\label{ThSimplex}
Let us denote by $\Delta$ the Laplace Beltrami operator on the Riemannian manifold $(S^n,g_{S^n})$, acting on 
$$\mathscr C^2_b(S^n):=\left\{ u\in \mathscr C^2(S^n):\max_{|\alpha|\leq 2}\sup_{x\in S^n}|\partial^\alpha u(x)|<\infty\right\},$$
where $S^n:=\{x\in \R^n: \sum_{i=1}^n x_i < 1,x_i>0\,\forall i=1,2,\dots,n\}$ and $g_{S^n}(x)$ is represented by the matrix
$$G_{S^n}(x):=\left[\begin{array}{ccccc}
x_1^{-1}     & 0&\dots& \dots &0\\
0& x_2^{-1}&0& \dots &0\\
\vdots  & \vdots    &\vdots&\vdots&\vdots\\
\vdots  & \vdots    &\vdots&\vdots&\vdots\\
0     & \dots&\dots& 0 &x_n^{-1}
\end{array}\right]+
\frac 1{1-\sum_{i=1}^n x_i}
\left[\begin{array}{ccccc}
1& 1&\dots& \dots &1\\
1& 1&\dots& \dots &1\\
1& 1&\dots& \dots &1\\
1& 1&\dots& \dots &1\\
1& 1&\dots& \dots &1
\end{array}\right].
$$
The operator $\Delta$ is symmetric and unbounded, it 
 has discrete spectrum 
$$\sigma(\Delta)=\{\lambda_s:=s(s+\frac{n-1}{2}: s\in \N\}$$
and the eigen-space of $\lambda_s$ is $\Span\{\psi_\alpha,|\alpha|=s\},$ where $\psi_\alpha$ (see Proposition \ref{orthogonalsystemsymplex})
are orthonormal polynomials with respect to the pluripotential equilibrium measure of the simplex
$$\mu_{S^n}:=\frac 1{\sqrt{(1-\sum_{i=1}^nx_i)\prod_{i=1}^n x_i}}\chi_{S^n}(x)\Vol_{\R^n}=\Vol_{g_{S^n}}.$$
Moreover, $\Delta$ can be closed to a self-adjoint operator (still denoted by $\Delta$) $\mathcal D(\Delta)\rightarrow L^2(S^n,g_{S^n})$ (having the same spectrum) where
$$\mathcal D(\Delta):=\left\{u\in L^2(S^n,g_{S^n}): \sum_{s=0}^\infty\lambda_s^{2}\sum_{|\alpha|=s} |\hat u_{\alpha}|^2<\infty\right\}\subset H^1(S^n,g_{S^n})$$
and $\hat u_\alpha$ is the Fourier coefficient $\int_{B^n} u \frac{\psi_\alpha}{\|\psi_\alpha^2|_{L^2(\mu_{S^n})}} d\mu_{S^n}.$    

The operator $\Delta^{1/2}$ has domain
$$\mathcal D(\Delta^{1/2}):=\left\{u\in L^2(S^n,g_{S^n}): \sum_{s=0}^\infty\lambda_s\sum_{|\alpha|=s} |\hat u_{\alpha}|^2<\infty\right\}= H^1(S^n,g_{S^n})$$
\end{theorem}

\begin{remark}[Refinement of Conjecture \ref{Conjecture}]
The computations performed in \cite{BoLeWa04}, \cite{BoLeWa08} and \cite{BlBoLe12} allow us to prove that Conjecture \ref{Conjecture} does hold for $E$ being the real ball and the real simplex. We remark that in this cases $\partial E$ is a real algebraic (possibly reduceble) set, thus one may  add such an assumption to Conjecture \ref{Conjecture}.
\end{remark}

\begin{remark}\label{pullbacks}
In order to better understand how the Baran metrics of the ball and the simplex look like, it is worth to recall their special relation with certain portion of the sphere.

Let us denote by $(\mathbb H^n_+,g_{\mathbb H^n_+})$ the \emph{upper unit hemisphere}, i.e., the Riemaniann manifold which can be obtained by intersecting the unit sphere $\mathbb S^{n}$ (thought as a sub-manifold of $\mathbb \R^{n+1}$ endowed with the euclidean metric) with the positive half space $\{\xi\in \R^{n+1}: \xi_{n+1}>0\}.$  The map $\pi: \mathbb H^n_+\rightarrow B^n$, $\pi(\xi):=(\xi_1,\dots,\xi_n)$ clearly is a one-to-one $\mathscr C^\infty$ map of manifolds. Therefore we can define a metric $g$ on $B^n$ by means of the pull-back operator with respect to $F:=\pi^{-1}$:
$$g(v,w):=F^*g_{\mathbb H^n_+}(v,w)=g_{\mathbb H^n_+}(dF v,dF w), \forall v,w\in TB^n.$$   

One can verify by direct computations that indeed $g\equiv g_{B^n}.$ 

Similarly, we can define the map $\Sqrt:S^n\rightarrow B^n\cap\{x\in \R^n: x_i>0,\forall i=1,\dots,n\}$, $\Sqrt(x):=(\sqrt{x_1},\sqrt{x_2},\dots,\sqrt{x_n}),$ and pull back by $\Sqrt$ on $S^n$ the Baran metric of the ball. Again this new metric indeed coincide with the Baran metric of the simplex.

In other words the maps $F$ and $sqrt$ are isometries of Riemaniann manifolds.
\end{remark}

Note that, since the manifolds $(B^n,g_{B^n})$ and $(S^n,g_{S^n})$ are isometric to certain portions of $\mathbb S^n$, the local differential and metric properties of this manifolds are the same of $\mathbb S^n.$
We recall that a Riemaniann manifold $(M,g)$ is termed Einstein when its metric tensor is a solution of the \emph{Einstein vacuum field equation}
\begin{equation}\label{vacuumfield}
\Ric=k g.
\end{equation}
Here 
$$\Ric_{i,j}:=\sum_{l=1}^n (\partial_l \Gamma_{j,i}^l-\partial_j \Gamma^l_{l,i})\,+\,\sum_{l,k=1}^n( \Gamma^l_{l,k}\Gamma^k_{j,i}- \Gamma^l_{j,k}\Gamma^k_{l,i} )$$
 is the Ricci tensor (written by means of the Christoffel symbols $\Gamma^i_{j,k}$) and $k>0.$
Since it is a well known fact that $(\mathbb S^n,g_{\mathbb S^n})$ is Einstein, we get the following proposition  as a consequence of Remark \ref{pullbacks}.  
\begin{proposition}
The unit ball and the unit simplex, endowed with their Baran metric respectively, are \emph{Einstein Manifolds}.
\end{proposition}

Since for all cases where the Baran metric is known to be Riemaniann it happens that it solves Equation \eqref{vacuumfield}, the following question naturally arises.
\begin{question}
Assume that $E$ is a Baran body in the sense of Definition \ref{defBB} below. Is it necessary for its Baran metric tensor to solve the Einstein vacuum field equation \eqref{vacuumfield}?
\end{question}

\begin{remark}
Recently, Zelditch \cite{Ze12} studied the spectral theory of the Laplace Beltrami operator on a real analytic Riemaniann manifold $M$ in connection  with the Pluripotential Theory  of the so called Bruhat-Whitney complexification $M_{\mathbb C}$ of $M.$ In particular, working under the assumption of ergodicity of the geodesic flow, \cite{Ze14,Ze07} present asymptotic results on the zero distribution of the eigenfunctions and series of functions with random Fourier coefficients. These results closely resemble the relation between the behaviour of zeros of orthogonal polynomials (or random polynomials) and the pluripotential equilibrium measure.

Even though our study is far to be as general as the context of the above references, in the author's opinion our result may be casted within this framework and offer concrete examples where explicit computations are performed. Indeed our Appendix \ref{appendix} exactly fits in the framework of \cite{Ze12}.      
\end{remark}

%\vskip 1cm
The paper is structured as follows. In Section \ref{SecPreliminaries} we furnish all the required definitions from Pluripotential Theory, Operator Theory and Differential Geometry. In Section \ref{SecProofs} we prove Theorems \ref{ThBall} and \ref{ThSimplex}, giving a precise spectral characterization of the involved Sobolev spaces. Finally, in the Appendix \ref{appendix} it is shown how to define the Baran metric on the sphere and its equivalence with the standard round metric.
\subsection*{Acknowledgements}
The ideas of the present paper surfaced during the open problems session of the  workshop \emph{Dolomites Research Week on Approximation} (\textsc{DRWA16}), held in Canazei (TN) Italy in September 2016. However, the content of the present paper has been deeply influenced by the note \cite{BoLeVi12} told by its second author during his visit at University of Padova in 2012. Therefore we would like to thank Norm Levenberg and the organizers of the conference and the Doctoral School of Mathematics of the University of Padova. Also, we would like to thank Prof. P.D. Lamberti for his helpfulness, Prof. M. Putti and M. Vianello for the useful discussions and the support they offered.

\section{Preliminaries and tools}\label{SecPreliminaries}
\subsection{The Pluripotential Theory framework}\label{subsectionPluripot}
Pluripotential Theory is the study of \emph{plurisubharmonic} functions, i.e., any upper-semicontinuous function $u:\Omega\rightarrow [-\infty, +\infty[$ being subharmonic along each one complex dimension  affine variety in $\Omega\subseteq_{open}\C^n.$ We use the operators $d:=\partial+\bar\partial$ and $d^c:=i(-\partial +\bar\partial)$, where
\begin{equation*}
\partial:= \sum_{j=1}^n\frac{\partial}{\partial z_j}\cdot dz_j,\;\;\;\; \bar\partial:= \sum_{j=1}^n\frac{\partial}{\partial \bar z_j}\cdot d\bar z_j.
\end{equation*}
The operator $\ddc$ is sometimes referred as complex Laplacian and correspond with the usual Laplacian (up to a scaling factor) when $n=1.$ 

Since $\ddc$ is a linear operator, one can consider $\ddc u$ for a $L^1_{\text{loc}}$ function in the sense of currents (distribution on the space of differential forms) and it turns out that, for an upper-semicontinuous function $u$, $\ddc u\geq 0$ if and only iff $u$ is plurisubharmonic. 

The \emph{complex Monge Ampere operator} $(\ddc{})^n$ is defined for $\mathscr C^2$ functions as
\begin{equation}
\label{MAdef}(\ddc{u})^n:=\ddc u \wedge \ddc u \wedge \dots \wedge \ddc u=c_n \det{(\ddc u)}d\Vol_{\C^n}.
\end{equation} 
Clearly trying to define wedge products of factors of the type $\ddc u$ for any plurisubharmonic function $u$ leads to serious difficulties due to the lack of linearity. Bedford and Taylor \cite{BeTa82} showed that the definition of equation \eqref{MAdef} can be extended to any  locally bounded plurisubharmonic function, being $(\ddc u)^n$ a positive Borel measure.

One may think to plurisubharmonic functions in $\C^n$ as ''the correct counterpart'' (see \cite[Preface]{Kli}) of subharmonic functions on $\C$, while harmonic functions should be replaced in this multi dimensional setting by \emph{maximal plurisubharmonic functions}, i.e., functions $u$ dominating on any subdomain $\Omega'$ any plurisubharmonic function $v$ such that $u\geq v$ on $\partial \Omega'.$ Locally bounded maximal plurisubharmonic functions satisfy $(\ddc u)^n=0.$ 

The multi dimensional counterpart of the Green function for the complement of a compact set $E$ is the \emph{pluricomplex Green function} (also known as Siciak-Zaharjuta extremal function) $V_E^*.$ Let $E\subset \C^n$ be a compact set, then we set 
\begin{equation}\label{efdef}
\begin{split}
&V_E(\z):=\sup\{u(\z), u\in\mathcal L(\C^n), u|_E\leq 0\},\\
&V_E^*(z):=\limsup_{\z\to z}V_E(\z).
\end{split}
\end{equation} 
Here $\mathcal L(\C^n)$ is the Lelong class of plurisubharmonic functions of logarithmic growth, i.e., $u(z)-\log|z|$ is bounded at infinity.

It is worth to recall that, as in the one dimensional case, due to \cite{Si81} (see also \cite{Kli}) we can express $V_E^*$ by means of polynomials $\wp(\C^n).$ That is 
\begin{align*}
&V_E(\z)=\sup\left\{\frac 1{\deg p}\log^+|p(\z)|, p\in\wp(\C^n), \|p\|_E\leq 1\right\}.
\end{align*}
The function $V_E^*$ is either identically $+\infty$ or a locally bounded plurisubharmonic function on $\C^n$, maximal on $\C^n\setminus E$ (i.e., $\ddcn{V_E^*}$ is a positive Borel measure with support in $E$) having  logarithmic growth at $\infty;$ if the latter case occurs we say that $E$ is \emph{non pluripolar}. In principle $V_E^*$ is only a upper semi-continuous function. When $V_E^*$ is continuous the compact set $E$ is said \emph{regular}. It is worth to recall that it turns out that $V_E^*$ is continuous if and only if $V_E^*$ identically vanishes on $E.$ We will treat only such a case in what follows.

For any non pluripolar compact set $E\subset \C^n$ the \emph{pluripotential equilibrium measure} of $E$ is defined as
\begin{equation}\label{pluripoteqmeasuredef}
\mu_E:=(\ddc V_E^*)^n,
\end{equation}
this is a Borel \emph{probability}  measure supported on $E.$ We stress that, since $\mu_E(E)=1$ for any non pluripolar set \cite{BeTa82}, the total mass of the measures (and volume forms) that we are going to deal with is not important. We avoid to introduce normalizing constant in the metrics to keep the notation simple.

Let $E$ be a real convex body, Baran showed that in such a case
\begin{equation}\label{baranmetricdef}
\delta_E(x,v):=\limsup_{t\to 0^+}\frac {V_E^*(x+i t v)}{t}
\end{equation}
exists for any $x\in\interior E,$ $v\in \R^n.$ We refer to $\delta_E(x,v)$ as the \emph{Baran metric} of $E.$ We refer the reader to \cite{BoLeWa08} for a study on the connections among this metric, polynomials inequalities and polynomial sampling.  The Baran metric defines in general a Finsler distance on $E$
$$d_E(x,y):=\inf\left\{\int_0^1 \delta_E(\gamma(s),\gamma'(s))ds, \gamma\in Lip([0,1],E),\gamma(0)=x,\gamma(1)=y\right\},$$
however it may happen that $\delta_E(x,v)$ is indeed Riemaniann, i.e.
$$\delta_E(x,v)=\sqrt{v^tG_E(x) v}$$
for a positive definite matrix $G_E(x).$ Note that $G_E(x)$ is then well defined by the \emph{parallelogram law}. More precisely we have
$$u^TG_E(x) v=\frac{\delta_E^2(x,u+v)-\delta_E^2(x,u-v)}{4}.$$

One of the possible motivation for the interest on the Baran metric comes from Approximation Theory. Indeed the \emph{Baran Inequality} (see \cite{Ba95,Ba92b} and \cite{BoLeVi12})
\begin{equation}\label{BaranInequality}
\begin{split}
&\frac{\left|\frac d{dt}p(\gamma((t))\right|}{\sqrt{1-p^2(\gamma(t))}}\leq \deg p\,\delta_E(\gamma(t),\gamma'(t)),\\
&\forall t\in[0,1], \gamma\in \mathscr C^1([0,1],E),p\in \wp^k(\C^n), \|p\|_E\leq 1,
\end{split}
\end{equation}
can be understood as a generalization of the classical Bernstein Inequality. For instance such inequality may be used to construct good sampling sets for polynomials, namely admissible meshes; see
\cite{CL08,Kr11,BoDeSoVi10,DePiSoVi15,Pi17}
%
%
%We can pick a differentiable arc $\gamma\in \mathscr C^1([0,1],E)$ on the convex set $E$ with $\gamma(0)=x_0,$ $\gamma(1)=x$ and write, for any polynomial $p$ such that $\|p\|_E\leq 1,$
%\begin{align*}
%&|p(x_0)-p(x)|\leq |\arccos p(x_0)- \arccos p(x)|\leq \left|\int_0^1 \frac{\frac d{dt}p(\gamma((t))}{\sqrt{1-p^2(\gamma(t))}}dt \right|\\
%\leq& \deg p\,\int_0^1 \delta_E(\gamma(t),\gamma'(t))dt,
%\end{align*}
%taking the infimum among such $\gamma$ we get $|p(x_0)-p(x)|\leq \deg p\,d_E(x,x_0).$
%
%It follows that, if the \emph{fill distance} $h(A_k,E)$ of the finite subset $A_k\subset E$ over $E$ with respect to the Baran distance is dominated by $c/k$ for some positive $c<1,$ then we have the sampling inequality
%\begin{equation*}
%\|p\|_E\leq \frac 1{1-c} \|p\|_{A_k}, \forall p\in \wp^k(\C^n),
%\end{equation*}
%that is \emph{the set $A_k$ is a norming set for the space of algebraic polynomials of degree at most $k.$}  We recall for the reader's convenience that the \emph{fill distance} of a subset $Y$ of a metric space $(X,d)$ over $X$  is defined as 
%$$h(Y,X):=\sup_{x\in X}\inf_{y\in Y}d(x,y).$$
%If $\{A_k\}$ can be constructed having cardinality that polynomially increases with respect to $k$, then such $\{A_k\}$ is termed an \emph{admissible mesh}, see \cite{CL08}. Admissible meshes have been fruitfully studied in the last decade both from the theoretical and computational point of view; \cite{Kr11,BoDeSoVi10,DePiSoVi15,Pi17}.   

We believe that the following definition is worth to be introduced.
\begin{definition}[Baran body]\label{defBB}
Let $\mathcal C$ denote either $\mathbb C^n$ or a irreducible algebraic variety of pure dimension $n,$ and let $\mathcal C_{\mathbb R}$ denote the real points of $\mathcal C.$ Let $E\subset \mathcal C_{\mathbb R}$ a compact fat\footnote{This mean that the closure in $\mathcal C_{\mathbb R}$ of the interior of $E$ in $\mathcal C_{\mathbb R}$ coincides with $E.$} non pluripolar set. If the Baran metric of $E$ is Riemaniann, then we term $E$ a \emph{Baran body}.    
\end{definition}

In \cite{BoLeVi12, BoLeWa08}, the Baran metrics of the real ball, real simplex are computed (see Theorem \ref{ThBall} and Theorem \ref{ThSimplex} above), showing in particular that they are Baran bodies. To the best author's knowledge, these are all the known examples of Baran compact sets in $\C^n.$ We offer a further instance of a Baran compact in Appendix \ref{appendix}: the real sphere as subset of the complexified sphere.

\subsection{Differential operators and Sobolev spaces on a Riemaniann manifold}
\subsubsection{Differential operators}
We recall that a liner connnection on a vector bundle $\pi:E\rightarrow M$ (built on the differentiable manifold $M$) is an application  (here $\mathcal E(M)$ is the space of smooth sections of the vector bundle $E$ and $\mathcal T(M)$ is the tangent bundle)
\begin{align*}
\nabla: &\mathcal T(M)\times \mathcal E(M) &\longrightarrow\;\;\;\;\;\;& \mathcal E(M)\\
        &(X,V)\;\;    &\longrightarrow\;\;\;\;\;\;& \nabla_X V
\end{align*}
such that it is $\mathscr C^\infty$-linear in $X$, $\R$-linear in $V,$ and for which holds the Liebnitz Rule $\nabla_X(f V)=V X(f)+ f\nabla_X(V),$ for any $f\in\mathscr C^\infty(M).$ In particular we have $\nabla_X f=X(f).$
 
Let $(M,g)$ be a (possibly non compact) Riemaniann manifold. It is well known that there exists a unique torsion-free linear connection on $\mathcal T(M)$ that is compatible with the metric $g$; namely the \emph{Levi-Civita connection}. Since we will deal only with such a connection we will still denote it by $\nabla.$ Indeed, the proof of the Levi Civita Theorem is fully constructive: the desired connection is expanded over a canonical basis and its coefficients, the \emph{Christoffel symbols} usually denoted by $\Gamma_{i,j}^k,$ are computed in terms the metric and its partial  derivatives. 

Note that, for a given $u\in \mathscr C^\infty(M),$ $\nabla u$ is a $(1,0)$ tensor field (i.e., point-wise it is a linear form) having the property that $( X,\nabla u)_g=\nabla_X u =X(u)$ and thus it can be written in local coordinates 
$$\nabla u=\sum_j g^{i j}\frac{\partial u}{\partial x_j}dx_j.$$
Here $(\cdot ,\cdot)_g$ is the canonical duality induced by $g$ and $g^{ij}$ are the components of the matrix representing $g^{-1}.$ Hence it is convenient to define the tangent vector 
$$(\grad u)_i :=\left(\sum_j g^{i j}\frac{\partial u}{\partial x_j}\right)_i,$$
namely the \emph{covariant gradient} of $u$, having the property that $( X,\nabla u)_g=\langle X,\grad u\rangle_g.$

The \emph{divergence} operator acting on $X\in \mathcal T(M)$ is defined by
$$\Div X:=\nabla \cdot X=\frac 1 {\sqrt{\det g}}\sum_i \partial_i( \sqrt{\det g} X^i).$$

Finally we can recall the definition of the \emph{Laplace Beltrami operator} $\Delta.$
\begin{equation}
\Delta u:=\Div(\grad u)=\frac 1 {\sqrt{\det g}}\sum_i \partial_i( \sqrt{\det g} (\grad u)_i).
\end{equation} 
\subsubsection{Sobolev Spaces}\label{SubsectionSobolev} 
Let $(M,g)$ be a Riemaniann manifold. Let us introduce on $\mathscr C^\infty(M)$ the norm
$$\|u\|_{1,2}:=\left( \int |u|^2 d\Vol_g\right)^{1/2}+ \left(\int |\grad u|^2d\Vol_g\right)^{1/2},$$
where $|\grad u|^2=\langle\grad u ,\grad u\rangle_g.$ Let us denote by $\mathscr C_{1,2}^\infty(M)$ the space $\{u\in \mathscr C^\infty(M), \|u\|_{1,2}<\infty\}.$

The \emph{Sobolev space} $H^1(M,g)$ is defined as the closure of  $\mathscr C_{1,2}^\infty(M)$ with respect to  $\|\cdot\|_{1,2}$ in the space of square integrable functions,  also we introduce the space $H^1_0(M,g)$ defined as the closure of $\mathscr C^\infty_c(M)$ in the same norm. Note that in principle $H^1_0(M,g)\subseteq H^1(M,g).$

An important fact about Sobolev spaces and manifold is that the above two spaces may coincide, that is
\begin{equation}
H^1_0(M,g)\equiv H^1(M,g)\label{sobolevequivalence}
\end{equation}
\emph{Our interest on this phenomena is mainly due to the fact that the Laplace operator does not need to be complemented with boundary conditions in such a case.}

Indeed, $H^1_0(M,g)\equiv H^1(M,g)$ for any complete Riemaniann manifold $M$; see \cite[Th. 3.1]{He99}. We recall  for the reader's convenience that a Riemaniann manifold $(M,g)$ is said to be complete if the metric space $(M,d_g)$ is complete, where \small
$$d_g(x,y):=\inf\left\{\int_0^1 \sqrt{\langle\gamma'(s),\gamma'(s)\rangle_{g(\gamma(s))}}ds, \gamma\in Lip([0,1],E),\gamma(0)=x,\gamma(1)=y\right\},$$
\normalsize
The Hopf-Rinow Theorem asserts that the completeness of $(M,g)$ is equivalent to the fact that any closed bounded subset of $M$ is compact.

We denote by $\mathscr C^\infty_b(M)$ the set uniformly bounded functions that have uniformly bounded partial derivatives of any order. Since for a complete manifold $\mathscr C^\infty_c(M)\subseteq \mathscr C^\infty_b(M)\subset H^1(M,g)$, it follows that  for any complete manifold $(M,g)$,  $\mathscr C^\infty_b(M)$  is dense in $H^1(M,g).$

Unfortunately, both $(B^n,g_{B^n})$ and $(S^n,g_{S^n})$ fail to be complete: it is very easy to construct a Cauchy sequence in $B^n$ not converging in $B^n.$ For instance take $\{x_k\}:=\cos{(2^{-k})}u$ for any unit vector $u\in \R^n$. Since  $d(x_{k},x_{l})\leq 2^{-\min(k,l)},$ this is a Cauchy sequence, however $x_k\to u\notin B^n.$ Nevertheless, one may wonder weather equation \eqref{sobolevequivalence} holds true in this instances. This fact indeed depends on finer properties of the manifolds than completeness. Namely, Masamune \cite{Ma05,Ma99} showed that equality \eqref{sobolevequivalence} holds if and only if the metric completion of $M$ lies in the category of manifolds with almost polar boundary.

We recall that the Riemaniann manifold $(M\cup \Gamma, g)$ with boundary $\Gamma$ is said to have \emph{almost polar boundary} if the outer capacity $\Capa(\Gamma)$ of $\Gamma$ vanishes. Here we use the notation $\Capa(A)$ for the Sobolev (outer) capacity of the Borel subset $A$ of $M\cup \Gamma,$ where for any open subset $O$ of $M\cup \Gamma$ we set
\begin{equation*}
\Capa(O):=\inf\{\|u\|_{1,2}, u\in\mathscr C^\infty_c(M\cup \Gamma), 0\leq u\leq 1, u|_O\equiv 1\}
\end{equation*}
and for for any Borel subset $S$ we set
$$\Capa(A):=\inf\{\Capa(O), A\subset O\}.$$
It is clear that one can replace $\mathscr C^\infty_c(M\cup \Gamma)$ by $H^1_0(M\cup \Gamma,g)$  in the definition of $\Capa(O)$ obtaining an equivalent definition.

At this stage we can observe that $\partial B^n$ fails the \emph{sufficient} condition (see \cite[Th. 7]{Ma05}) to be polar
\begin{equation}\label{sufficientalmostpolar}
\liminf_{\epsilon\to 0^+}\frac{\log \Vol\left(\{x\in B^n: d(x,\partial B^n)<\epsilon  \}\right)}{\log\epsilon}\geq 2.
\end{equation}
Here equality case is considered since $\partial B^n$ itself is a manifold (see \cite[Th. 7]{Ma05}).

Let us denote by $N_\epsilon$ the set $\{x\in B^n: d(x,\partial B^n)<\epsilon  \}$, we have 
$N_\epsilon=B^n \setminus (\cos \epsilon)\cdot B^n$, moreover
$$\Vol (N_\epsilon)=\pi \beta(1/2,n/2,1-(\cos \epsilon)^2).$$
Here $\beta(a,b,z)$ denotes the Incomplete Beta Function $\int_0^z t^{a-1}(1-t)^{b-1}dt.$
Hence
$$\frac{\Vol N_\epsilon}{\epsilon^2}\sim \frac{\Vol N_\epsilon}{1-(\cos \epsilon)^2} \frac{1-(\cos \epsilon)^2}{\epsilon^2}\sim 2\frac{\Vol N_\epsilon}{1-(\cos \epsilon)^2},\;\;\text{ as } \epsilon\to 0^+.$$
Note that
$$\liminf_{\epsilon\to 0^+}\frac{\Vol N_\epsilon}{1-(\cos \epsilon)^2}=\lim_{z\to 0^+}\frac{\beta(1/2,n/2,z)}{z}=\lim_{z\to 0^+}z^{-1/2}(1-z)^{n/2-1}=+\infty.$$
Thus we have $\liminf_{\epsilon\to 0^+}\frac{\Vol N_\epsilon}{\epsilon^2}=+\infty$ that in particular implies $\frac{\log \Vol N_\epsilon}{\log \epsilon}<2$ for any $\epsilon<\epsilon_0.$

Since the condition \eqref{sufficientalmostpolar} is not fulfilled by $\partial B^n$ nor $\partial S^n$ we wonder if the ball and the simplex, endowed with their Baran metrics, are not manifold with almost polar boundary. Indeed this is the case, as stated in the following proposition. However, since these conclusions are obtained as a consequence of Theorem \ref{ThBall} and Theorem \ref{ThSimplex} respectively, we cannot use them in the proof of such theorems. 
\begin{proposition}\label{propnotalmostpolar}
The manifolds $(B^n,g_{B^n})$ and $(S^n,g_{S^n})$ are not manifold with almost polar boundary and
\begin{equation}
H^1(B^n,g_{B^n})\neq H^1_0(B^n,g_{B^n})\;,\;\;\;\;H^1(S^n,g_{S^n})\neq H^1_0(S^n,g_{S^n}).
\end{equation}
\end{proposition} 
\begin{remark}
We warn the reader that $H^1(B^n,g_{B^n})\neq H^1_0(B^n,g_{B^n})$ does not imply in general that the eigenvalue problem $\Delta u =\lambda u$ is not well posed when we do not impose any boundary condition. The motivation depends on the following proposition which allows us to write the weak formulations \eqref{ballsymmetry} and \eqref{simplexsymmetry} of the Laplace Beltrami operator used in the proofs of Theorem \ref{ThBall} and \ref{ThSimplex} which is based on $\mathscr C^\infty_b$ functions (for which the boundary terms appearing in the integration by parts formulas we use vanish).
\end{remark}
\begin{proposition}\label{propdensity}
Let $(M,g)$ be $(B,g_B)$ or $(S,g_S)$. The space $\mathscr C^\infty_b(M)$ is dense in $\mathscr C^\infty_{1,2}(M)$ with respect to the norm $\|\cdot\|_{1,2}.$ Thus $\mathscr C^\infty_b(M)$ is dense in $H^1(M,g).$
\end{proposition}
Before proving Proposition \ref{propdensity} we need this two technical Lemmata whose proofs are omitted since it is sufficient to check the statements by easy direct computations.
\begin{lemma}[The inverse Baran metric of the ball]\label{Lemmaballinversemetric}
Let us denote by $G_{B^n}^{-1}$ the inverse of the matrix $G_{B^n}$ which represents the Baran metric of the $n$-dimensional ball. Then we have
\begin{equation}\label{ballinversemetric}
G_{B^n}^{-1}(x):=\left[\begin{array}{cccccc}
1- x_1^2& -x_1 x_2& -x_1 x_3&\dots&\dots&x_1 x_n\\
-x_2 x_1& 1-x_2^2 & -x_2 x_3&\dots&\dots&x_2 x_n\\
\vdots &\vdots &\vdots &\vdots &\vdots &\vdots \\
-x_n x_1&\dots & \dots&-x_n x_{n-2}&-x_n x_{n-1}&1- x_n^2
\end{array}\right].
\end{equation}  
The matrix $G_{B^n}^{-1}(x)$ has eigenvalues $\{1,1-|x|^2\}$, where the eigen-space of $1$ is the tangent space at $x$ to the sphere of radius $|x|$ and centred at zero, while the eigen-space of $1-|x|^2$ is the Euclidean normal to such a sphere at $x.$ 
\end{lemma}

\begin{lemma}[The inverse Baran metric of the simplex]\label{lemmasimplexinversemetric}
Let us denote by $G_{S^n}^{-1}$ the inverse of the matrix $G_{S^n}$ which represents the Baran metric of the $n$-dimensional simplex. Then we have
\begin{equation}\label{simplexinversemetric}
G_{S^n}^{-1}(x):=\left[\begin{array}{ccccc}
(1- x_1)x_1& -x_1 x_2&\dots&\dots&x_1 x_n\\
-x_2 x_1& (1-x_2)x_2 &\dots&\dots&x_2 x_n\\
\vdots &\vdots &\vdots&\vdots &\vdots \\
-x_n x_1&\dots & \dots&-x_n x_{n-1}&(1- x_n)x_n
\end{array}\right].
\end{equation}
Moreover we have
\begin{equation}\label{relationinversegramian}
G_{S^n}^{-1}(x)=  diag(\sqrt{x_1},\dots,\sqrt{x_n})\,G_{B^n}^{-1}(\sqrt{x_1},\dots,\sqrt{x_n})\,diag(\sqrt{x_1},\dots,\sqrt{x_n}).
\end{equation}
\end{lemma}

\begin{proof}[Proof of Proposition \ref{propdensity}]
Let us start by considering the case $M=B^n\subset \R^n.$ We denote by $\mathbb S^n$ the $n$ dimensional unit real sphere endowed with the standard \emph{round metric} $g_{\mathbb S^n}$ and we introduce the embedding map 
$$E:\mathscr C^\infty_{1,2}(M)\rightarrow H_{\text{even}}^1(\mathbb S^n,g_{\mathbb S^n}),$$
where 
\small
$$ E[f]\left(x_1,\dots,x_n,\pm \sqrt{1-\sum_{i=1}^n x_i^2}\right):=\frac 1 {\sqrt 2}\begin{cases} f(x_1,\dots,x_n)&,\sum_{i=1}^n x_i^2\neq 1\\ {\limsup}_{B\ni \xi\to x} f(\xi_1,\dots,\xi_n)&, \sum_{i=1}^n x_i^2= 1\end{cases}$$  
\normalsize and \small
$$H_{\text{even}}^1(\mathbb S^n,g_{\mathbb S^n}):=\left\{g\in H^1(\mathbb S^n,g_{\mathbb S^n}),\;g(x_1,\dots,x_{n+1})=g(x_1,\dots,-x_{n+1})\, \Vol_{\mathbb S^n}-a.e.\right\}.$$
\normalsize We claim that $E$ is an \emph{isometry} of Hilbert spaces. 

Before proving such a claim we stress that this would conclude the proof for the case of the ball. For, by standard mollification we can construct a sequence $\{\hat f_k\}$ of function in $\mathscr C^\infty(\mathbb S^n)$ converging to $E[f]$ in $H^1(\mathbb S^n,g_{\mathbb S^n}).$ To ensure that $\tilde f_k \in H_{\text{even}}^1(\mathbb S^n,g_{\mathbb S^n})$ we set $\tilde f_k:=(\hat f_k(x_1,\dots,x_{n+1})+\hat f_k(x_1,\dots,-x_{n+1}))/2.$
Finally define $\{f_k\}:=\{E^{-1}[\tilde f_k]\}$ and note that the claim above implies that $f_k\to f$ in $H^1(B^n,g_{B^n}).$ 

We stress that, while the injectivity of $E$ is trivial, one needs to notice that the global boundedness of $\tilde f_k$ together with its derivatives ensure that $E^{-1}[\tilde f_k]$ is a well defined element of $\mathscr C^\infty_{1,2}(M)$ which in particular is in $\mathscr C^\infty_b(M).$

Let us go back to prove that $E$ is an isometric embedding. For simplicity we work in the easy case of $n=2$, the general case can be proved in a completely equivalent way. Consider the spherical coordinates 
$$\left(\begin{array}{c}x_1\\x_2\\x_3\end{array}\right)=\left(\begin{array}{c}\cos \theta \cos \phi\\ \cos \theta \sin \phi\\\sin \theta\end{array}\right).$$  
We recall that the round metric represented in this coordinates is
$$g_{\mathbb S^2}(\theta, \phi):=\left[\begin{array}{cc}1 & 0 \\0 & \cos^2 \theta  \end{array}\right]$$
and the corresponding volume form can be written $d \Vol_{\mathbb S^2}=\cos \theta d\theta d\phi.$  It follows that, for any $h\in H^1(\mathbb S^2,g_{\mathbb S^2})$ we have 
$$\|h\|_{H^1(\mathbb S^2)}^2=\int_0^{2\pi}\int_{-\pi/2}^{\pi/2}\left(|h|^2+ |\partial_\theta h|^2+\frac{|\partial_\phi h|}{\cos^2\theta}\right)\cos \theta d\theta d\phi.$$
To compute $\|E[f]\|_{H^1(\mathbb S^2)}$ we perform the change of variables suggested by the first two components of the spherical coordinates, i.e.,
$$(x_1,x_2)\mapsto(\cos \theta \cos \phi, \cos \theta \sin \phi).$$

It is easy to verify by a direct computation that
\begin{align*}
&\|E[f]\|_{H^1(\mathbb S^2)}^2\\
=& \int_0^{2\pi}\int_{-\pi/2}^{\pi/2}\left(|E[f]|^2+ |\partial_\theta E[f]|^2+\frac{|\partial_\phi E[f]|}{\cos^2\theta}\right)\cos \theta d\theta d\phi\\
=& 2\int_0^{2\pi}\int_{0}^{\pi/2}\left(|E[f]|^2+ |\partial_\theta E[f]|^2+\frac{|\partial_\phi E[f]|}{\cos^2\theta}\right)\cos \theta d\theta d\phi\\
=&2 \int_{B^2} \left(\frac{|f|^2} 2 + (1-x_1^2-x_2^2) \frac{|\partial_n f|^2} 2 +\frac{|\partial_t f|^2} 2\right)\frac 1{\sqrt{1-x_1^2-x_2^2}}dx_1dx_2\\
=&\int_{B^2} \left(|f|^2 + |\grad f|_{g_{B^2}}^2\right) d\Vol_{g_{B^2}}\\
=&\|f\|_{H^1(B,g_{B^2})}^2.
\end{align*}

Let us now consider the case $M=S^n.$ We introduce the embedding map 
$$F:\mathscr C^\infty_{1,2}( S^n,g_{S^n})\rightarrow V,$$
where 
$$V:=\{f\in \mathscr C^\infty_{1,2}(B^n,g_{B^n}), f(\xi_1,\dots,\xi_j,\dots,\xi_n)=f(\xi_1,\dots,-\xi_j,\dots,\xi_n),\forall j\in \{1,\dots,d\}\}$$
and
$$ F[h]\left(\xi_1,\dots,\xi_n\right):= \frac 1 2 h(\xi_1^2,\dots,\xi_n^2).$$
Again if the closure of $F$ to $H^1(S^n,g_{S^n})$ is an isometric embedding  we are done, since, for any given target function $h\in H^1(S^n, g_{S^n})$ we can pull back to $\mathscr C^\infty_b(S^n)$ any sequence of $\mathscr C^\infty_b(B^n)$ approximations to $F[h].$

To this aim, we introduce the partition $Q_1,\dots,Q_{2^n}$ of $[-1,1]^n$ given by the coordinates hyperplanes, we denote by $T:S^n\rightarrow B^n$ the map $(\xi_1,\dots,\xi_n)\mapsto (\xi_1^2,\dots,\xi_n^2)=x$ and we 
notice that, for any $f\in \mathscr C^\infty (S^n)$, we have 
$$\int_{B^n\cap Q_j} f\circ T d\Vol_{B^n}=\frac 1{2^n}\int_{S^n} f d\Vol_{S^n}.$$
Finally we compute
\begin{align*}
&\|F[h]\|_{H^1(B^n,g_{B^n})}\\
=&\sum_{j=1}^{2^n}\int_{B^n\cap Q_j} \left(|F[h](\xi)|^2 +|\grad F[h](\xi)|_{g_{B^n}}^2\right) d\Vol_{B^n}(\xi)\\
=&\frac 1 4\sum_{j=1}^{2^n}\int_{B^n\cap Q_j}\left(|h\circ T(\xi)|^2+Dh^t\circ T JT^t g_{B^n}^{-1} JT Dh\circ T(\xi)\right) d\Vol_{B^n}(\xi)\\
=&\frac 1 {4\cdot 2^n}\sum_{j=1}^{2^n}\int_{T^{-1}(B^n\cap Q_j)}\left(|h(x)|^2 +Dh^t (JT^t g_{B^n}^{-1} JT)\circ T^{-1} Dh(x)\right) d\Vol_{S^n}(x)\\
=&\frac 1 {4 \cdot 2^n}2^n\int_{T^{-1}(B^n\cap Q_1)}\left(|h(x)|^2 +Dh^t (JT^t g_{B^n}^{-1} JT)\circ T^{-1} Dh(x)\right) d\Vol_{S^n}(x).
\end{align*} 
Since, due to equation \eqref{relationinversegramian},
$$(JT^t g_{B^n}^{-1} JT)\circ T^{-1}=4 diag(\xi)g_{B^n}^{-1}diag(\xi)\Big |_{\xi=\sqrt x}=4 g_{S^n}(x)$$
we conclude that $\|F[h]\|_{H^1(B^n,g_{B^n})}= \|h\|_{H^1(S^n,g_{S^n})}.$ In view of the above reasoning this concludes the proof.
\end{proof}

\subsection{Unbounded linear operators on Hilbert spaces, some tools.}
We need to recall some concepts from Operator Theory that allow a more precise and compact formulation of our results. A \emph{linear operator} on a Banach space $\mathscr B$ is a couple $(\mathcal D_{\mathscr B}(T),T)$, where $\mathcal D_{\mathscr B}(T)$ is a dense linear subspace of $\mathscr B$ and $T$ is a linear map $\mathcal D_{\mathscr B}(T)\rightarrow \mathscr B.$

Let $(\mathcal D_{\mathscr B}(T),T)$ be a linear operator. If for any sequence $\{f_n\}$ in $\mathcal D_{\mathscr B}(T)$ such that 
\begin{itemize}
\item $\|f_n\to f\|_{\mathscr B}\to 0$ for some $f\in\mathscr B,$
\item there exists $g\in\mathscr B$ with $\|T f_n-g\|_{\mathscr B}\to 0$
\end{itemize} 
it follows that $f\in \mathcal D_{\mathscr B}(T)$ and $Tf=g$, then the operator $T$ is said to be \emph{closed}. 
If $\mathscr B$ is not finite dimensional, the notion of spectrum and set of eigenvalues are not coinciding. More precisely, we denote by $\sigma(T)$ the \emph{spectrum} of $T$
$$\sigma(T):=\{z\in \C: T-z\mathbb I \text{  is not  invertible}\}.$$
Instead, $\lambda$ is an eigenvalue of $T$ if there exists an element $f\in \mathscr B$ such that $Tf=\lambda f.$

If an operator $T$ is not closed we may try to find and extension of it, i.e., $(\tilde T, \mathcal D_{\mathscr B}(\tilde T))$ such that $\mathcal D_{\mathscr B}(\tilde T)\supset \mathcal D_{\mathscr B}(T)$ and $\tilde Tf=Tf$ for any $f\in \mathcal D_{\mathscr B}(T).$ If we can find such an extension in the category of closed operators, then $T$ is said to be \emph{closable} and its minimal closed extension $\overline T$ is termed \emph{the closure} of $T.$ 

Now we replace the Banach space $\mathscr B$ by an Hilbert space $\mathscr H$, clearly the above terminologies are still well defined, since any Hilbert space is in particular Banach.

If for any $f, g\in \mathcal D_{\mathscr H}(T)$ we have $\langle Tf,g\rangle_{\mathscr H}=\langle f,Tg\rangle_{\mathscr H}$, then the operator $T$ is said to be \emph{symmetric}. It is a very useful fact that \emph{any symmetric operator is closable to a symmetric operator}. Again, if $\mathscr H$ is infinite dimensional, one must pay attention to the difference among symmetric and self-adjoint operators. 

The adjoint $T^*$ of the operator $T$ is defined by the relation
$$\langle Tf,g\rangle_{\mathscr H}=\langle f,T^*g\rangle_{\mathscr H}, \forall f\in \mathcal D_{\mathscr H}(T), g\in \mathcal D_{\mathscr H}(T^*),$$
where 
$$\mathcal D_{\mathscr H}(T^*):=\left\{g\in \mathscr H:\exists h\in \mathscr H\text{ such that }\langle Tf,g\rangle_{\mathscr H}=\langle f,h\rangle_{\mathscr H}, \forall f\in \mathcal D_{\mathscr H}(T)    \right\}.$$
Clearly, we term $T$ self-adjoint when the two domains indeed coincide.

The proofs of our results, besides the explicit computations, rely on the following theorem which collects some classical results of Operator Theory; see for instance \cite[Ch. 1 and Ch. 4]{Da95}.
\begin{theorem}\label{ThOpTh}
Let $T$ be a linear non negative unbounded operator on the Hilbert space $(\mathscr H, \|\cdot\|)$ with domain $\mathcal D(T)$. Assume that
\begin{enumerate}[a)]
\item $T$ is symmetric,
\item It has discrete real spectrum $\sigma(T)=\{\lambda_j\}_{j\in \N}$ diverging to $+\infty.$
\end{enumerate} 
Then
\begin{enumerate}[i)]
\item the closure $\bar T$ of $T$ is a self-adjoint unbounded operator (i.e., $T$ is essentially self-adjoint),
\item $\sigma( \bar T)=\sigma(T)$,
\item the domain of $\bar T$ is
\begin{equation}
\mathcal D(\bar T)=\{u\in \mathscr H:\;\sum_{j=1}^{\infty}\lambda_j^2|\hat u_j|^2<\infty\} 
\end{equation}
\item the quadratic form
$$Q(u):=\langle T^{1/2}u,T^{1/2}u\rangle_{\mathscr H}$$
has domain
\begin{equation}
\mathcal D(Q)=\{u\in \mathscr H:\;\sum_{j=1}^{\infty}\lambda_j|\hat u_j|^2<\infty\} 
\end{equation}
which is complete in the norm
$$|\|u\||:=\sqrt{Q(u)}+\|u\|_{\mathscr H}.$$
\end{enumerate}
\end{theorem}

\section{Proofs}\label{SecProofs}
The strategy of the proofs of Theorems \ref{ThBall} and \ref{ThSimplex} is to show that the conditions $a)$ and $b)$ of Theorem \ref{ThOpTh} holds for $T$ being the Laplace Beltrami operator (with respect to the considered metric), then to conclude applying Theorem \ref{ThOpTh}. This will be done by considering the weak formulation of the Laplace Beltrami operator and performing explicit computations on a suitable orthogonal system.   

\subsection{Orthogonal polynomials in $L^2_{\mu_{B^n}}$}
The following family of orthogonal functions on the unit ball has been first introduced in the Approximation Theory framework, indeed the formula we will use is a special case of orthogonal polynomials for certain radial weight functions; see \cite[Ch. 5]{DuXu14}.  

\begin{figure}[h]
\caption{The first ten basis functions $\phi_\alpha$ of Proposition \ref{orthogonalsystemball} generates $\wp^3(B^2).$ }
\label{basis}
\begin{tabular}{ccccc}
\includegraphics[scale=0.15]{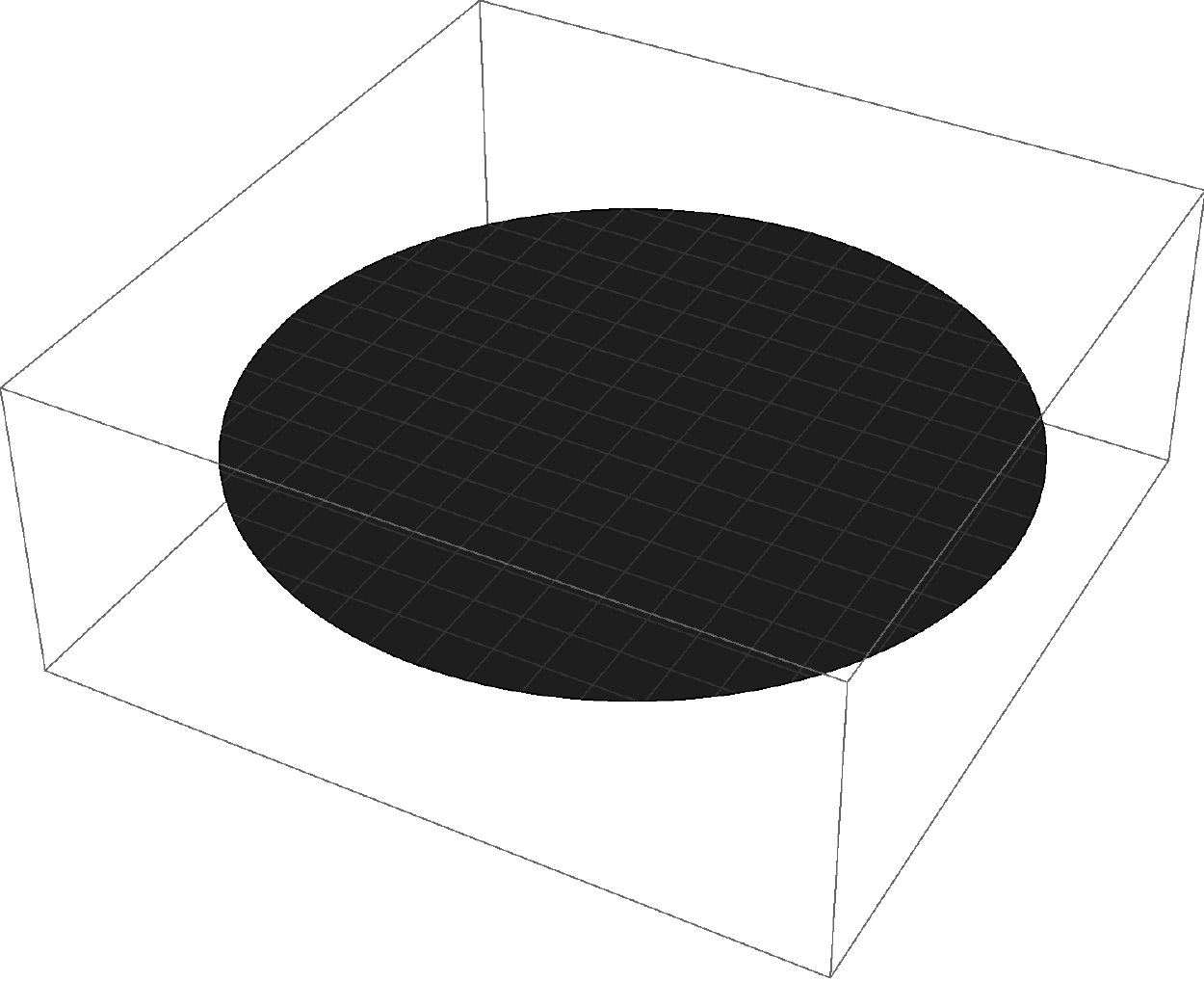}&
\includegraphics[scale=0.15]{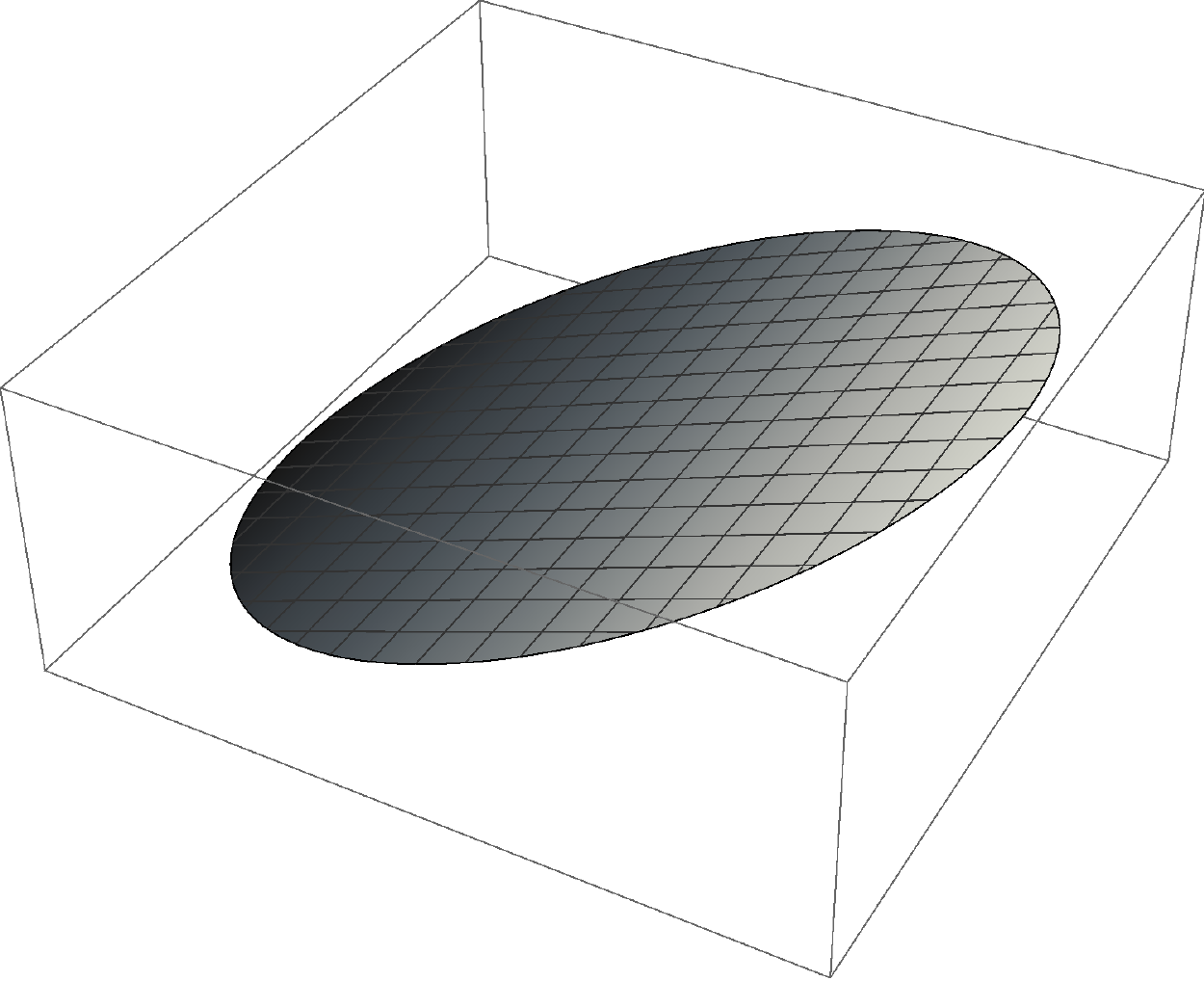}&
\includegraphics[scale=0.15]{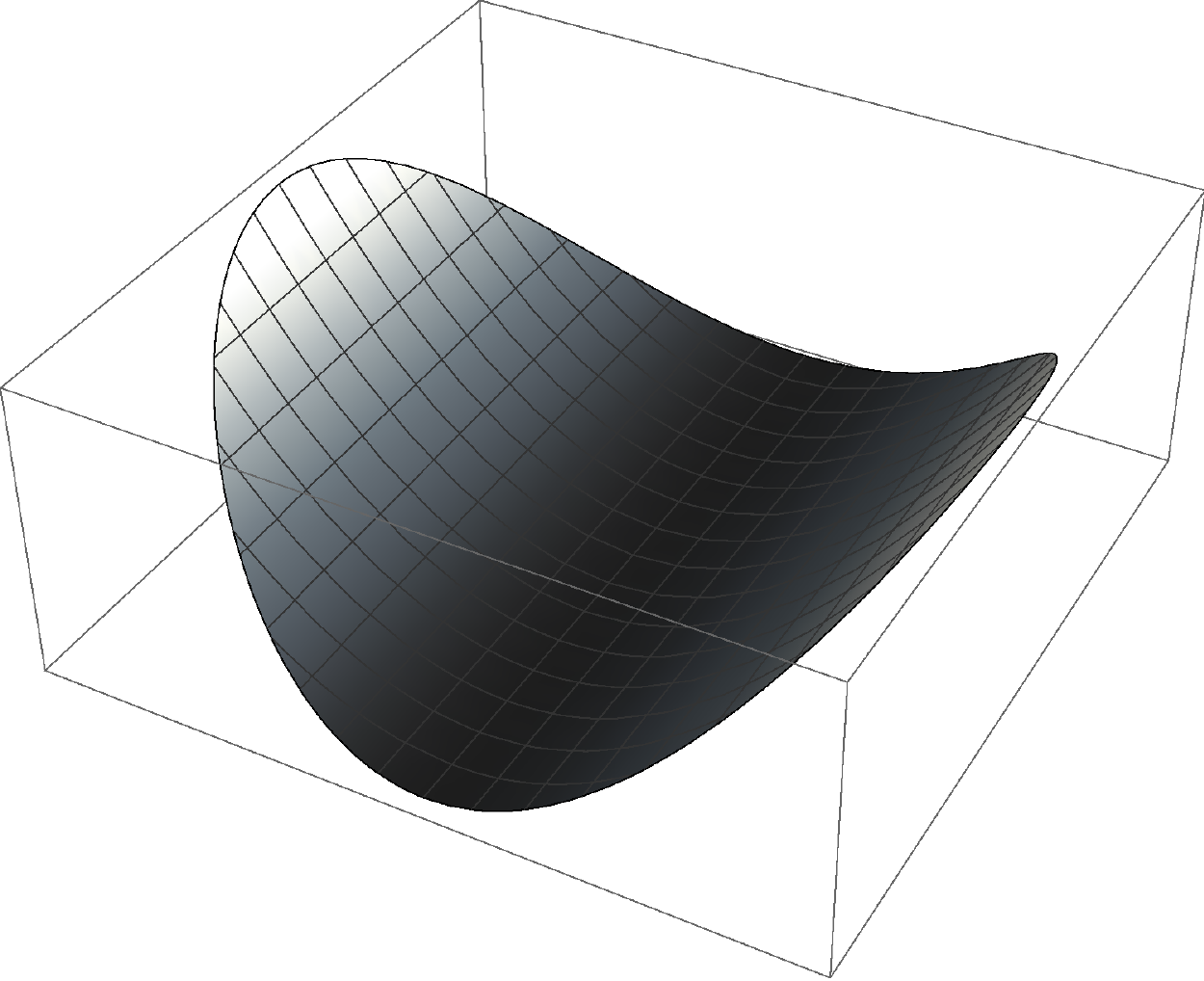}&
\includegraphics[scale=0.15]{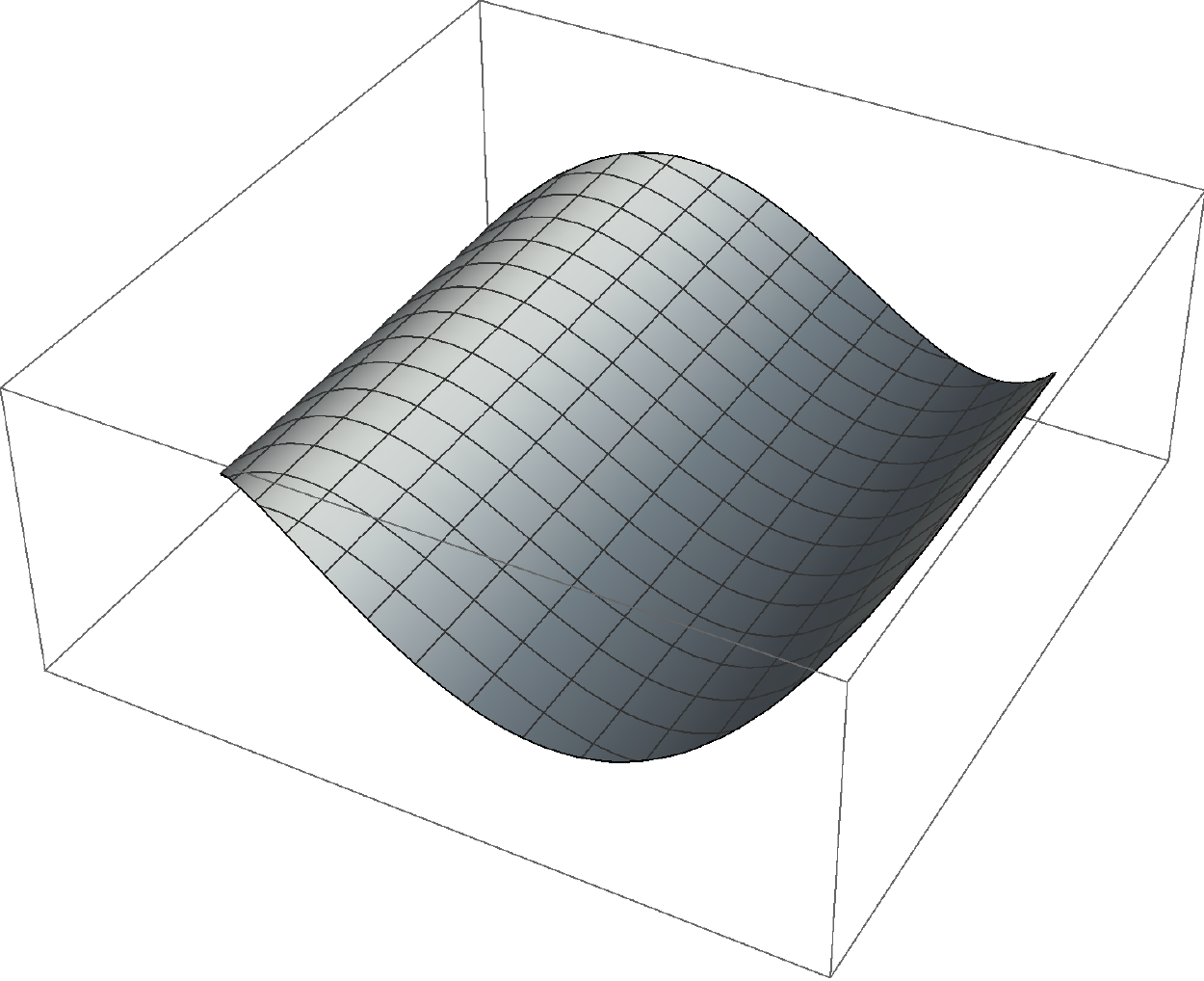}&
\includegraphics[scale=0.15]{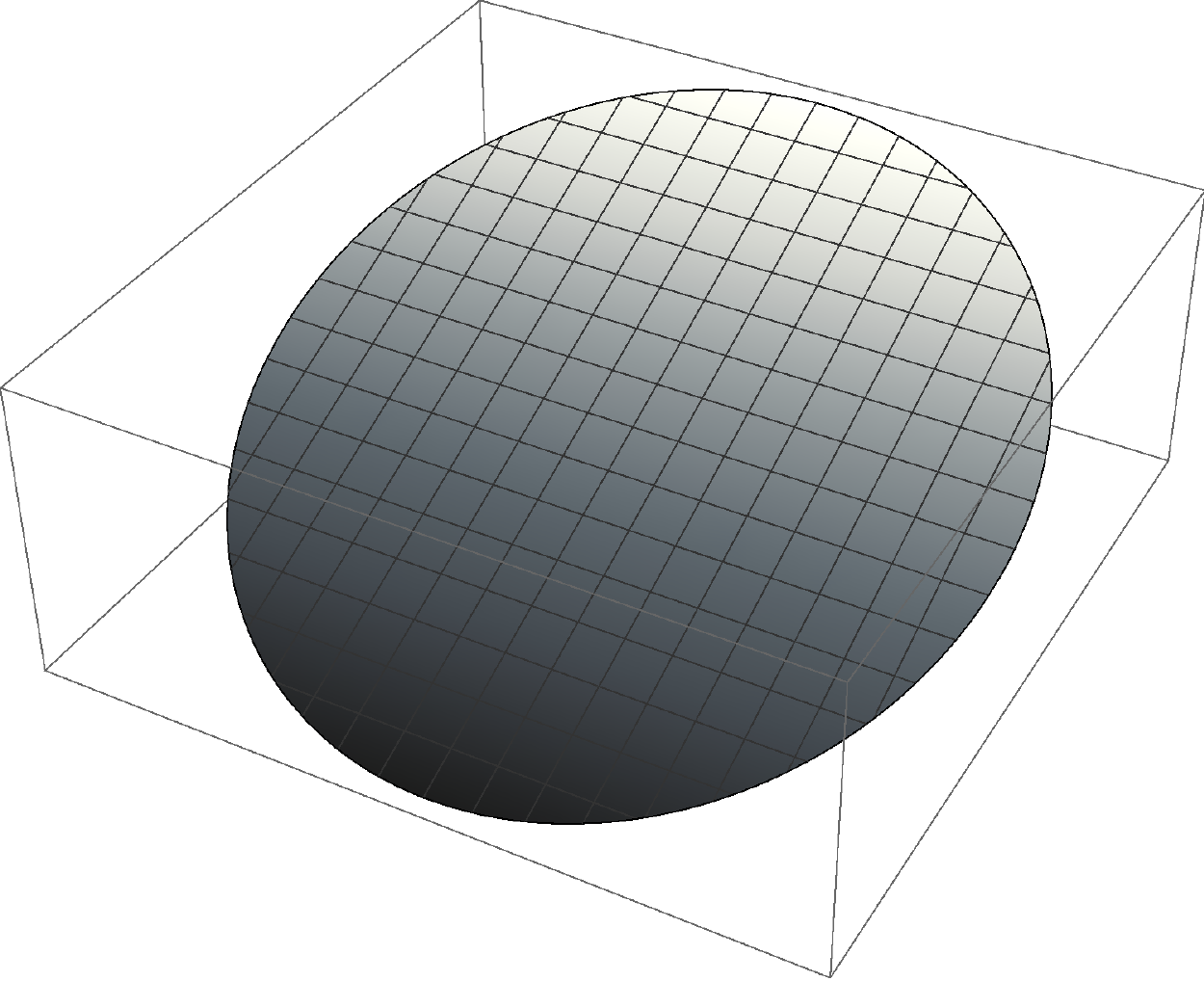}\\
\includegraphics[scale=0.15]{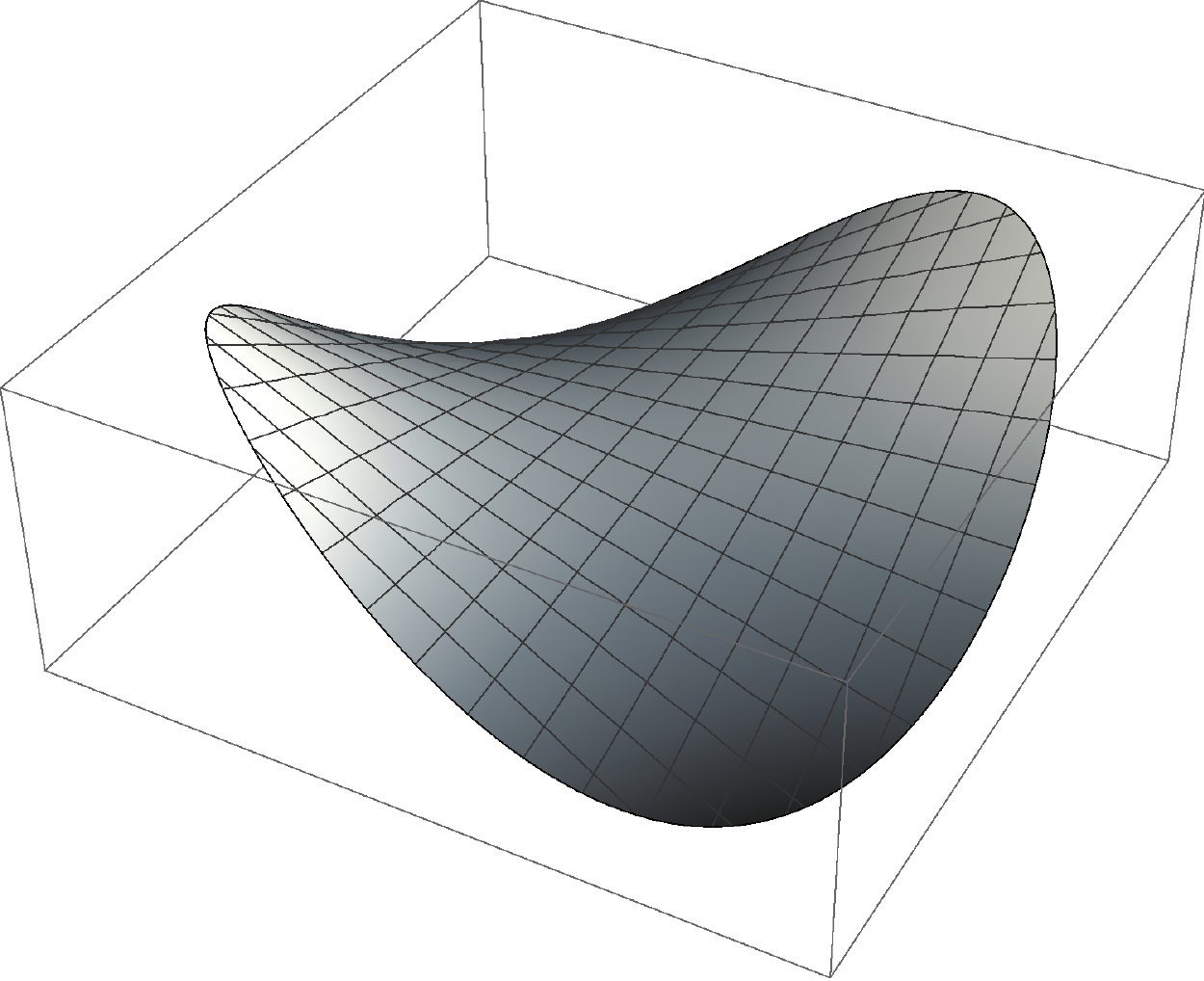}&
\includegraphics[scale=0.15]{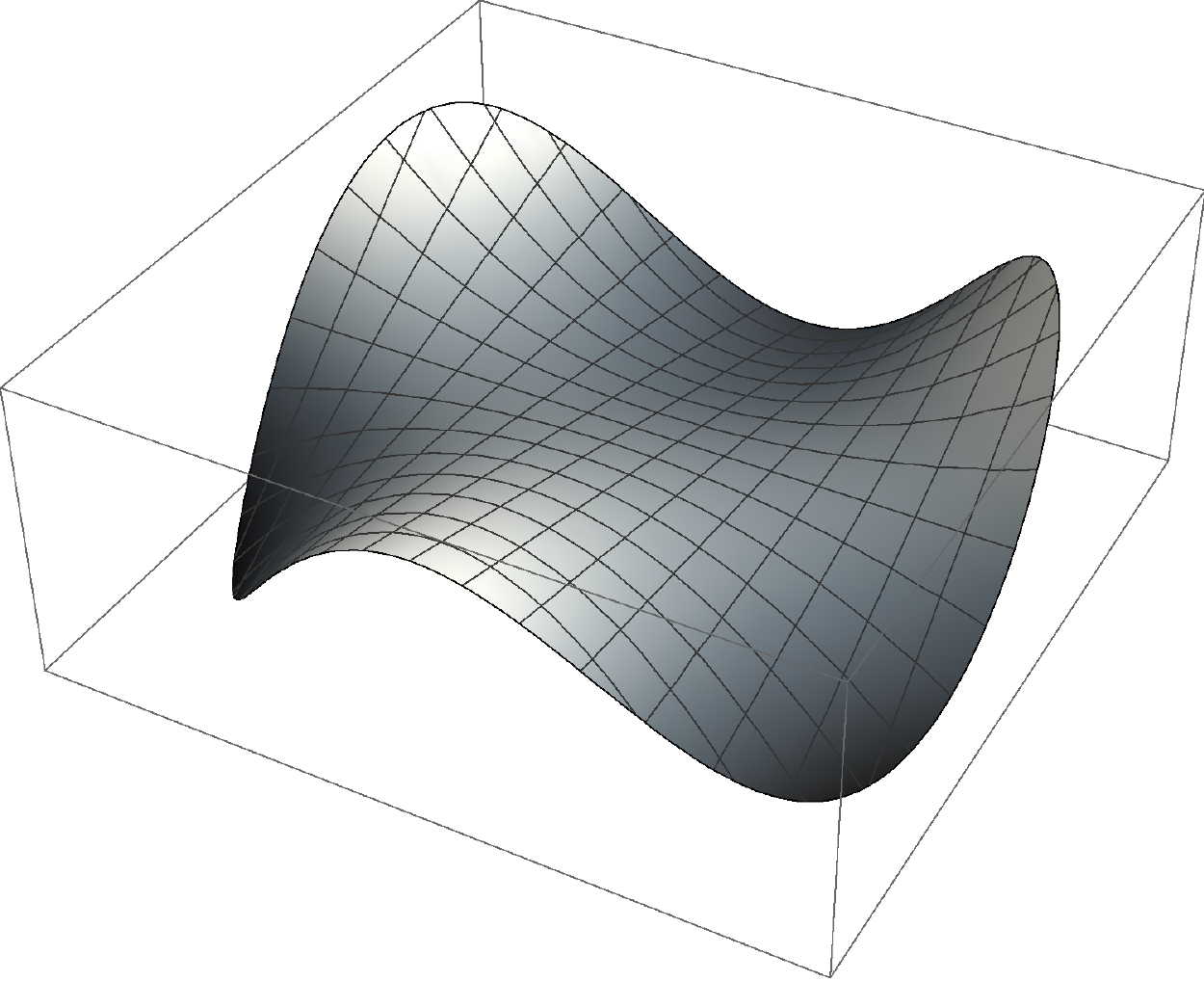}&
\includegraphics[scale=0.15]{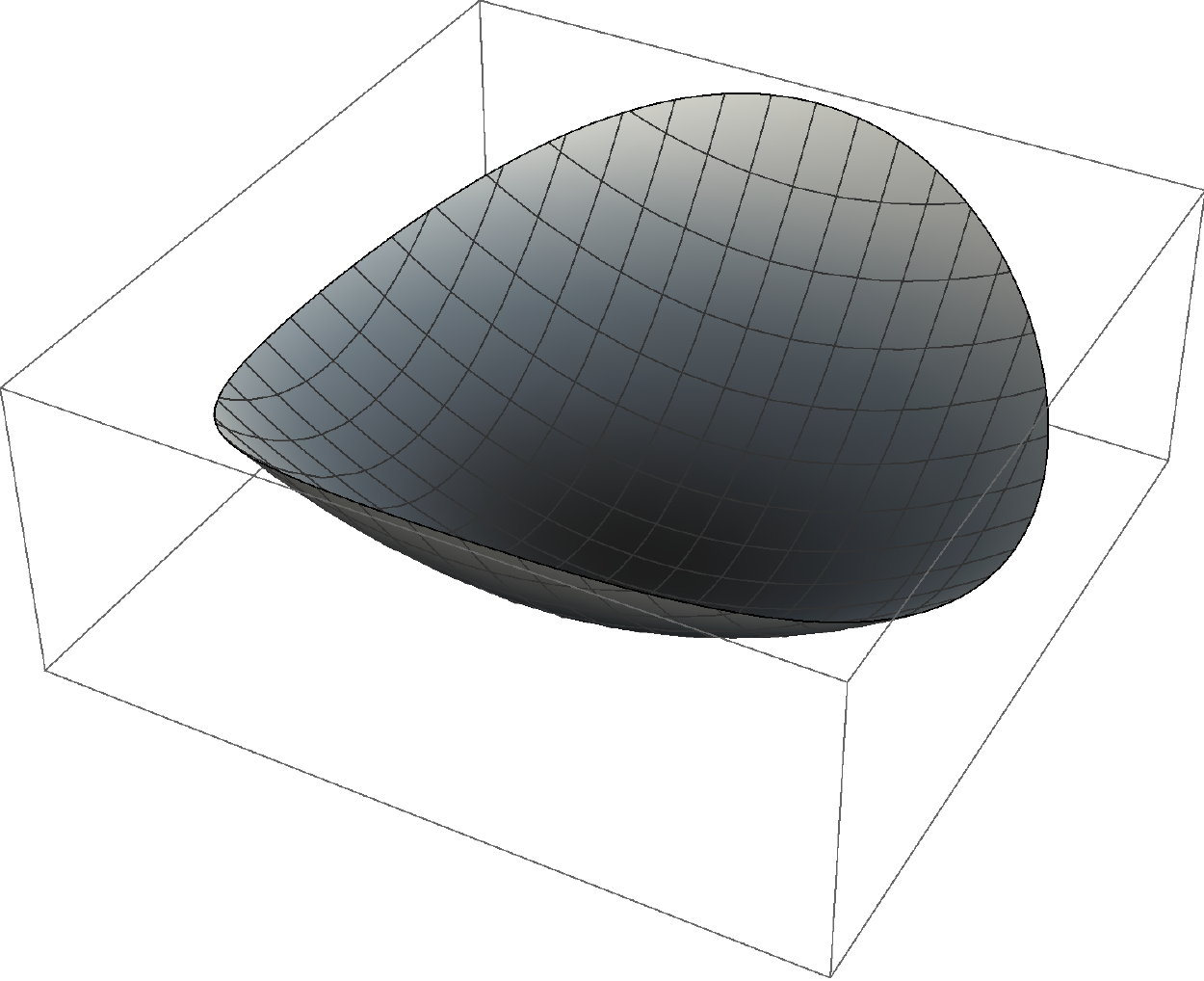}&
\includegraphics[scale=0.15]{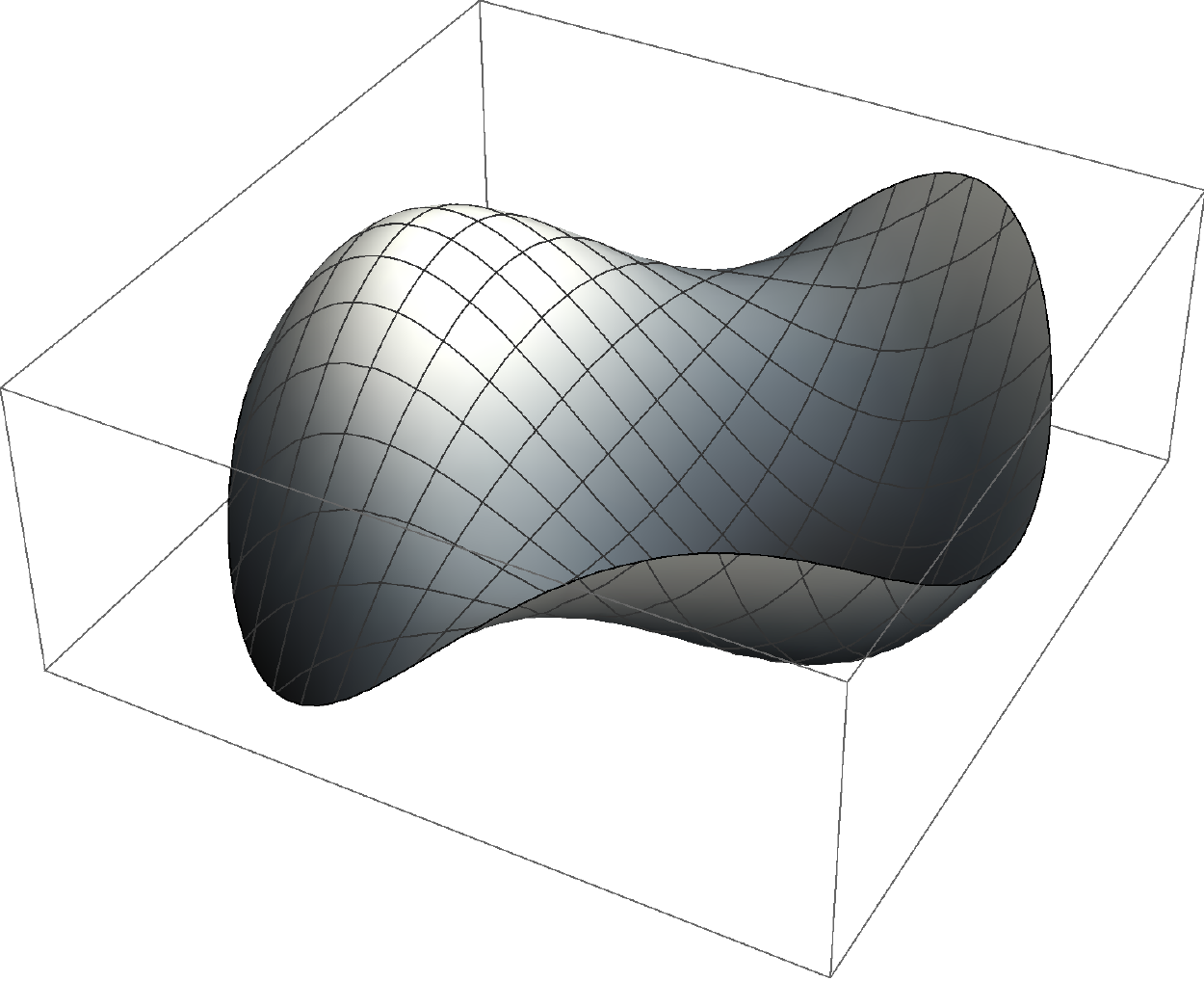}&
\includegraphics[scale=0.15]{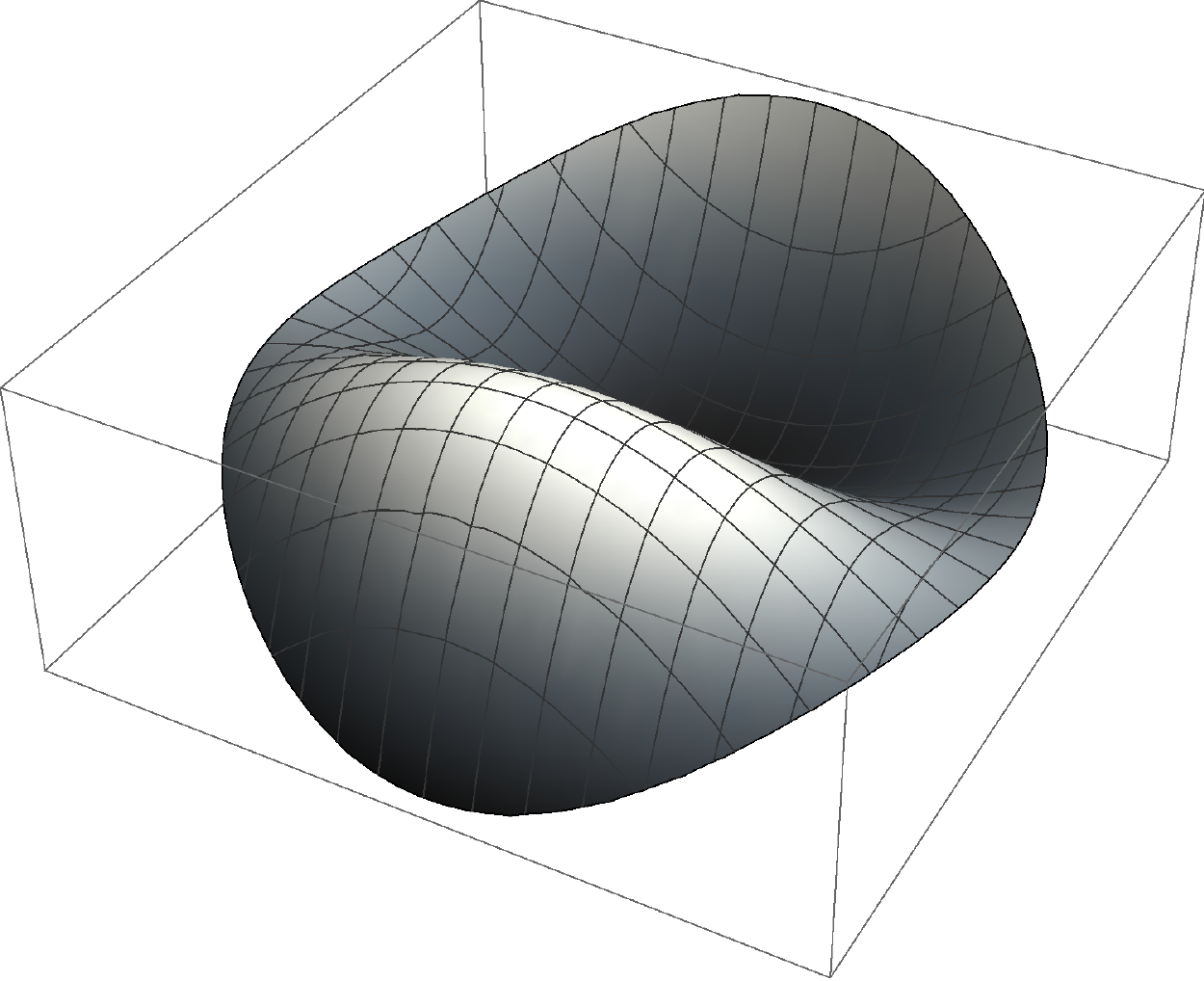}
\end{tabular}
\end{figure}

\begin{proposition}[\cite{DuXu14}]\label{orthogonalsystemball}
Let us set for any $\alpha\in \N^{n}$\small
\begin{equation}
\phi_\alpha:=T_{\alpha_n}\left(\frac{x_n}{\sqrt{1-\sum_{k=1}^{n-1}x_k^2}}  \right)\prod_{j=1}^{n-1}(1-\sum_{k=1}^{j-1} x_k^2)^{\alpha_j/2}C_{\alpha_j}^{\lambda_j}\left( \frac{x_j}{\sqrt{1-\sum_{k=1}^{j-1} x_k^2}} \right) ,
\end{equation}\normalsize
where $T_k$ is the Chebyshev polynomial of degree $k$, $\lambda_j:=\frac{n-j}2+\sum_{k=j+1}^n\alpha_k$ and $C^s_t$ denote the monic Gegenbauer polynomials of degree $t$ (i.e., $C^s_t:=J^{s-1/2,s-1/2}_t$ and $J_t^{\alpha,\beta}$ is the monic Jacobi polynomial). 

The set $\{\phi_\alpha: \alpha\in \N^{n}\}$ is a dense orthogonal system in $L^2({B^n},g_{B^n})$ and
\begin{equation}\label{normsontheball}
\|\phi_\alpha\|^2_{L^2({B^n},g_{B^n})}=\|T_{\alpha_n}\|_{-1/2,-1/2}^2\prod_{j=1}^{n-1}\|C_{\alpha_j}^{\lambda_j}\|_{\alpha_j-1/2,\alpha_j-1/2}^2,
\end{equation}
where $\|f\|_{a,b}:=\left(\int_{-1}^1|f(t)|^2(1-t)^a(1+t)^bdt\right)^{1/2}.$
\end{proposition}

Note that the density of the linear subspace $\Span\{\phi_\alpha: \alpha\in \N^{n}\}$ in $H^1({B^n},g_{B^n})$ follows by Proposition \ref{propdensity}.
\subsection{Proof of Theorem \ref{ThBall}}
we \textbf{warn the reader} that we will denote throughout this section by $D f$ the \emph{Euclidean gradient of }$f.$ 
\begin{proof}[Proof of Theorem \ref{ThBall}]
We start showing that $\Delta$ acting on $\mathscr C^2_b(B^n)$ is a symmetric operator. Namely, for any $u,v\in \mathscr C^2_b(B^n),$ we have
\small
\begin{equation}\label{ballsymmetry}
\int_{B^n}u\Delta  v d\Vol_{B^n}=-\int_{B^n}\langle \grad u,\grad v\rangle_{g_{B^n}}d\Vol_{B^n} =\int_{B^n}v\Delta u d\Vol_{B^n}.
\end{equation}
\normalsize 
In order to prove this formula we perform two integrations by parts.
\begin{align*}
&-\int_{B^n}v\Delta u d\Vol_{B^n}=-\int_{B^n} \Div(\sqrt{\det g_{B^n}} G_{B^n}^{-1}D u)v dx\\
=&\lim_{r\to 1} -\int_{{B^n}_r} \Div(\sqrt{\det g_{B^n}} G_{B^n}^{-1}D u)v dx\\
=&\lim_{r\to 1} \left(\int_{{B^n}_r}  D v^TG_{B^n}^{-1}D u \sqrt{\det g_{B^n}}dx-\int_{\partial {B^n}_r}  \nu^T G_{B^n}^{-1}D u \sqrt{\det g_{B^n}}d\sigma\right)\\
=& \int_{B^n}\langle \grad u,\grad v\rangle_{g_{B^n}}d\Vol_{B^n}-\lim_{r\to 1}\int_{\partial {B^n}_r}  \nu^T G_{B^n}^{-1}D u \sqrt{\det g_{B^n}}d\sigma\\
=& -\int_{B^n} \Div(\sqrt{\det g_{B^n}} G_{B^n}^{-1}D v)u dx+\\
&\;\;\;\;\;\;\;\; \lim_{r\to 1}\int_{\partial {B^n}_r}  u\nu^T G_{B^n}^{-1}D v \sqrt{\det g_{B^n}}d\sigma -\lim_{r\to 1}\int_{\partial {B^n}_r} v \nu^T G_{B^n}^{-1}D u \sqrt{\det g_{B^n}}d\sigma\\
=&-\int_{B^n}u\Delta v d\Vol_{B^n}+ \lim_{r\to 1}\int_{\partial {B^n}_r}  u\nu^T G_{B^n}^{-1}D v \sqrt{\det g_{B^n}}d\sigma\\
&\;\;\;\;\;\;\;\;\;\;\;\;\;\;\;\; -\lim_{r\to 1}\int_{\partial {B^n}_r} v \nu^T G_{B^n}^{-1}D u \sqrt{\det g_{B^n}}d\sigma
\end{align*}
Here $\nu$ is the (euclidean) unit outward normal to $\partial {B^n}_r:=\{x\in \R^n:|x|\leq r\}.$
 
The proof of \eqref{ballsymmetry} is concluded if we show that 
$$ \lim_{r\to 1}\int_{\partial {B^n}_r} v \nu^T G_{B^n}^{-1}D u \sqrt{\det g_{B^n}}d\sigma=0$$
for any $u,v\in \mathscr C^2_b({B^n}).$ For, simply observe (see Lemma \ref{Lemmaballinversemetric}) that $\nu$ is an eigenvector of $G_{B^n}^{-1}$ of eigenvalue $(\det g)^{-1}|_{|x|=r}=(1-|x|^2)$, thus we have
\begin{align*}
&\lim_{r\to 1}\int_{\partial B_r} v \nu^T G_B^{-1}D u \sqrt{\det g}d\sigma=\lim_{r\to 1} \sqrt{1-r^2}\int_{\partial B_r} v \partial_\nu u d\sigma\\
\leq&\lim_{r\to 1}\sqrt{1-r^2}C(u,f)=0.
\end{align*}
This shows that condition $a)$ of Theorem \ref{ThOpTh} holds for $\Delta.$ To conclude the proof we need to show that $b)$ holds as well, i.e.,  there exists a $L^2(B^n,g_{B^n})$-orthogonal system in $\mathscr C^2_b(B^n)$ dense in $L^2(B^n,g_{B^n})$ made of eigenfunctions of $\Delta$ such that the corresponding eigenvalues are a positive diverging sequence. We claim that such an orthogonal system is, indeed $\{\phi_\alpha, \alpha\in \N^n\},$ see Proposition \ref{orthogonalsystemball}.

For the sake of readability we present here the case $n=2$, which leads to slightly easier notation and computations with respect to the general one. However, all the elements of the proof of the general case are presented in such a simplified exposition. To easy the notation we denote $B^2$ by $B.$

The orthogonal basis of Proposition \ref{orthogonalsystemball} reads as
$$\phi_{s,k}(x,y):=(1-x^2)^{k/2}J_{s-k}^{k,k}(x) T_k\left( \frac y{\sqrt{1-x^2}}\right), 0\leq s\leq k\in \N,$$
where we denoted by $J_m^{\alpha,\beta}$ the $m$-th Jacobi orthogonal polynomial with respect to $(1-x)^\alpha(1+x)^\beta.$ We need to verify
$$\langle -\Delta \phi_{s,k},\phi_{m,l}\rangle_{L^2(B,g_B)}=\lambda_{s,k}\delta_{s,m}\delta_{k,l}$$
Since $\phi_{s,k}$ are elements of $\mathscr C^{\infty}_b( B)$ we can use the above weak formulation \eqref{ballsymmetry} to get
\begin{align*}
&\langle -\Delta \phi_{s,k},\phi_{m,l}\rangle_{L^2(B,g_B)}=\int_BD \phi_{s,k}^T G_B^{-1}D \phi_{m,l}\sqrt{\det g_B}dxdy.
\end{align*}
Let us introduce a change of variables
$$(x,z)\mapsto \Psi(x,z):=(x,z\sqrt{1-x^2})=(x,y).$$
We denote by $J\psi$ the Jacobian matrix of $\Psi$ so we get
\begin{align*}
&\int_B D f_1^T G_B^{-1}D f_2 \sqrt{\det g}dx dy=\\ 
=&\int_{-1}^1\int_{-1}^1D(f_1\circ\Psi)^T J\Psi^{-T}G_B^{-1}J\Psi^{-1} D(f_2\circ \Psi) dx \frac {dz}{\sqrt{1-z^2}}\\
=&\int_{-1}^1\int_{-1}^1D(f_1\circ\Psi)^T \left[\begin{array}{cc} 1-x^2&0\\0&\frac{1-z^2}{1-x^2}\end{array}\right] D(f_2\circ \Psi) dx \frac {dz}{\sqrt{1-z^2}}.
\end{align*}
Note that not only $\Psi$ is a change of variables that diagonalizes $G_B^{-1}$, also it has the property of giving to the basis functions $\phi_{s,k}$ a tensor product structure. Indeed we have
$\phi_{s,k}\circ\Psi(x,z)=(1-x^2)^{k/2}J_{s-k}^{k,k}(x) T_k( z)$, thus
\begin{align*}
&\int_{-1}^1\int_{-1}^1D(\phi_{}s,k\circ\Psi)^T \left[\begin{array}{cc} 1-x^2&0\\0&\frac{1-z^2}{1-x^2}\end{array}\right] D(\phi_{m,l}\circ \Psi) dx \frac {dz}{\sqrt{1-z^2}}\\
=&\int_{-1}^{1}\partial_x[(1-x^2)^{k/2}J_{s-k}^{k,k}(x)] \partial_x[(1-x^2)^{l/2}J_{m-l}^{l,l}(x)](1-x^2)dx\;\cdot\\
&\;\;\;\;\;\;\;\int_{-1}^1 T_k(z)T_l(z)\frac {dz}{\sqrt{1-z^2}}\;+\\
&\;\;\;\int_{-1}^{1}(1-x^2)^{(k+l)/2-1}J_{s-k}^{k,k}(x)J_{m-l}^{l,l}(x)]dx\;\cdot\\
&\;\;\;\;\;\;\;\int_{-1}^1 \partial_zT_k(z)\partial_z T_l(z)\sqrt{1-z^2}dz
\end{align*} 
It is well known that 
$$\int_{-1}^1 T_k(z)T_l(z)\frac {dz}{\sqrt{1-z^2}}= 2^{\delta_k\delta_l}\pi/2\delta_{l,k}.$$
Also one has $T_k'=kU_{k-1},$ where $U_k$ are orthogonal Chebyshev ppolynomials of the second kind, i.e., 
$$\int_{-1}^1 U_k(z)U_l(z)\sqrt{1-z^2}dz= \pi/2\delta_{l,k}.$$
Using such orthogonality and differentiation relations in the above computation we get
\begin{align}
&\int_{-1}^{1}\partial_x[(1-x^2)^{k/2}J_{s-k}^{k,k}(x)] \partial_x[(1-x^2)^{l/2}J_{m-l}^{l,l}(x)](1-x^2)dx\;\cdot\notag \\
&\;\;\;\;\;\;\;\int_{-1}^1 T_k(z)T_l(z)\frac {dz}{\sqrt{1-z^2}}\;+\notag \\
&\;\;\;\;\;\;\;\int_{-1}^{1}(1-x^2)^{(k+l)/2-1}J_{s-k}^{k,k}(x)J_{m-l}^{l,l}(x)]dx\;\cdot\notag \\
&\;\;\;\;\;\;\;\int_{-1}^1 \partial_zT_k(z)\partial_z T_l(z)\sqrt{1-z^2}dz\notag \\
=&\frac\pi 2\delta_{l,k}\Big(\int_{-1}^{1}\partial_x[(1-x^2)^{k/2}J_{s-k}^{k,k}(x)] \partial_x[(1-x^2)^{k/2}J_{m-k}^{k,k}(x)](1-x^2)dx\;\cdot2^{\delta_k}\notag\\
&\;\;\;+k^2\int_{-1}^{1}(1-x^2)^{k-1}J_{s-k}^{k,k}(x)J_{m-k}^{k,k}(x)dx\Big)\label{intermstep} .
\end{align}
Now we note that
\begin{align*}
&\int_{-1}^{1}\partial_x[(1-x^2)^{k/2}J_{s-k}^{k,k}(x)] \partial_x[(1-x^2)^{k/2}J_{m-k}^{k,k}(x)](1-x^2)dx\\
=&\int_{-1}^{1}\partial_x[J_{s-k}^{k,k}(x)]\partial_x[J_{m-k}^{k,k}(x)](1-x^2)^{k+1}dx\\
&\;\;\;\;+\int_{-1}^{1}-kx\partial_x[J_{s-k}^{k,k}(x)J_{m-k}^{k,k}(x)](1-x^2)^{k}dx\\
&\;\;\;\;+k^2\int_{-1}^{1}x^2J_{s-k}^{k,k}(x)J_{m-k}^{k,k}(x)(1-x^2)^{k-1}dx,
\end{align*}
integration by parts in the second term leads to
\begin{align*}
&\int_{-1}^{1}\partial_x[(1-x^2)^{k/2}J_{s-k}^{k,k}(x)] \partial_x[(1-x^2)^{k/2}J_{m-k}^{k,k}(x)](1-x^2)dx\\
=&\int_{-1}^{1}\partial_x[J_{s-k}^{k,k}(x)]\partial_x[J_{m-k}^{k,k}(x)](1-x^2)^{k+1}dx\\
&\;\;\;\;-2k^2\int_{-1}^{1}x^2J_{s-k}^{k,k}(x)J_{m-k}^{k,k}(x)(1-x^2)^{k-1}dx\\
&\;\;\;\;+k\int_{-1}^{1}J_{s-k}^{k,k}(x)J_{m-k}^{k,k}(x)(1-x^2)^{k}dx\\
&\;\;\;\;+k^2\int_{-1}^{1}x^2J_{s-k}^{k,k}(x)J_{m-k}^{k,k}(x)(1-x^2)^{k-1}dx.
\end{align*}
We plug this last identity in \eqref{intermstep} so we get
\begin{align*}
&\langle -\Delta_B \phi_{s,k},\phi_{m,l}\rangle_{L^2(B,g_B)}\\
=&\frac\pi 2\delta_{l,k}2^{\delta_k}\Big( \int_{-1}^{1}\partial_x[J_{s-k}^{k,k}(x)]\partial_x[J_{m-k}^{k,k}(x)](1-x^2)^{k+1}dx\\
&\;\;\;\;+k^2\int_{-1}^{1}(1-x^2)J_{s-k}^{k,k}(x)J_{m-k}^{k,k}(x)(1-x^2)^{k-1}\\
&\;\;\;\;+k\int_{-1}^{1}J_{s-k}^{k,k}(x)J_{m-k}^{k,k}(x)(1-x^2)^{k}dx\Big)\\
=&\frac\pi 2\delta_{l,k}2^{\delta_k}\Big( \int_{-1}^{1}\partial_x[J_{s-k}^{k,k}(x)]\partial_x[J_{m-k}^{k,k}(x)](1-x^2)^{k+1}dx\\
&\;\;\;\;+k(k+1)\int_{-1}^{1}J_{s-k}^{k,k}(x)J_{m-k}^{k,k}(x)(1-x^2)^{k}dx\Big).
\end{align*}
The last term in the sum vanishes for any $m\neq s$, this follows from the orthogonality of Jacobi polynomials. When instead $m=s$ we have (see for instance \cite{DuXu01})
$$k(k+1)\int_{-1}^{1}(J_{s-k}^{k,k}(x)(x))^2(1-x^2)^{k}dx=\frac{k(k+1)2^{2k+1}(s!)^2}{(2s+1)(s+k)!(s-k)!}.$$

For the first term, we recall that $\frac d{dx}J_{s-k}^{k,k}=\frac{s+k+1}2J_{s-k-1}^{k+1,k+1}$, Hence, using again the orthogonality, we get 
\begin{align*}
&\int_{-1}^{1}\partial_x[J_{s-k}^{k,k}(x)]\partial_x[J_{m-k}^{k,k}(x)](1-x^2)^{k+1}dx\\
=&\left(\frac{s+k+1}2\right)^2\int_{-1}^{1}(J_{s-k-1}^{k+1,k+1})^2(1-x^2)^{k+1}dx\\
=&(s+k+1) \frac{2^{2k+1}(s!)^2}{(2s+1)(s+k)!(s-k-1)!}.
\end{align*}
We finally computed
\begin{align*}
&\langle -\Delta_B \phi_{s,k},\phi_{m,l}\rangle_{L^2(B,g_B)}\\
=&\frac\pi 2\delta_{l,k}2^{\delta_k}\delta_{s,m}\frac{2^{2k+1}(s!)^2}{(2s+1)(s+k)!(s-k)!}\Big( k(k+1)+(s+k+1)(s-k)\Big)\\
=&s(s+1)\frac\pi 2\delta_{l,k}\delta_{s,m}2^{\delta_k}\frac{2^{2k+1}(s!)^2}{(2s+1)(s+k)!(s-k)!}\\
=&s(s+1)\|\phi_{s,k}\|_{L^2(B,g_B)}^2\delta_{l,k}\delta_{s,m}.
\end{align*}
Here the last line is due to Proposition \ref{orthogonalsystemball}.
\end{proof}

\subsection{Orthogonal polynomials in $L^2_{\mu_{S^n}}$}
\begin{proposition}[\cite{DuXu14}]\label{orthogonalsystemsymplex}
Let us set for any $\alpha\in \N^{n}$ and $x\in S^n$
\begin{equation}
\psi_\alpha(x):=\prod_{j=1}^n\left(1-\sum_{k=1}^{j-1}x_k \right)^{\alpha_j}J_{\alpha_j}^{a_j,-1/2}\left(\frac{2x_j}{1-\sum_{k=1}^{j-1}x_k}-1 \right),
\end{equation}\normalsize
where $J_m^{a,b}$ is the $m$-th Jacobi polynomial of parameters $a,b$ and
$$a_j:=2\sum_{k=1}^{\min(n,j+1)}\alpha_k+\frac{n-j-1}{2}.$$

The set $\{\psi_\alpha: \alpha\in \N^{n}\}$ is a dense orthogonal system in $L^2(S^n,g_{S^n}).$
% and
%\begin{equation}
%\|\psi_\alpha\|^2_{L^2(S^n,g_{S^n})}=\dots
%\end{equation}
\end{proposition}
This result (see Th. 8.2.2 in \cite{DuXu14}) plays a key role in our proof.
\begin{theorem}[\cite{DuXu14}]\label{simplexeigenfunctions}
Let us introduce the differential operator
\begin{equation}\label{differentialproperty}
\mathcal D f:=\sum_{i=1}^nx_i \partial_{i,i}^2 f-2\sum_{1\leq i<j\leq n}x_i x_j \partial_{i,j}^2 f +\frac 1 2\sum_{i=1}^n(1-(n+1)x_i)\partial_i f.
\end{equation}
Then we have
\begin{equation}
\mathcal D\psi_\alpha=|\alpha|\left(|\alpha|+\frac{n+1}2\right)\psi_\alpha.
\end{equation}
\end{theorem}

\subsection{Proof of Theorem \ref{ThSimplex}}
\begin{proof}[Proof of Theorem \ref{ThSimplex}] 
Let us introduce, see Figure \ref{trianglefig}, the following notations for $\epsilon>0$
\begin{align*}
S^n_\epsilon &:=\left\{x\in S^n: x_i>\epsilon,(1-\sum_{k=1}^nx_i)>\epsilon\right\},\\
T^{n,0}_\epsilon&:=\left\{x\in \partial S^n_\epsilon: (1-\sum_{k=1}^nx_i)=\epsilon\right\},\\
T^{n,i}_\epsilon&:=\{x\in \partial S^n_\epsilon: x_i=\epsilon\},\;i=1,\dots, n.
\end{align*}
\begin{figure}
\begin{center}
\includegraphics[scale=0.1]{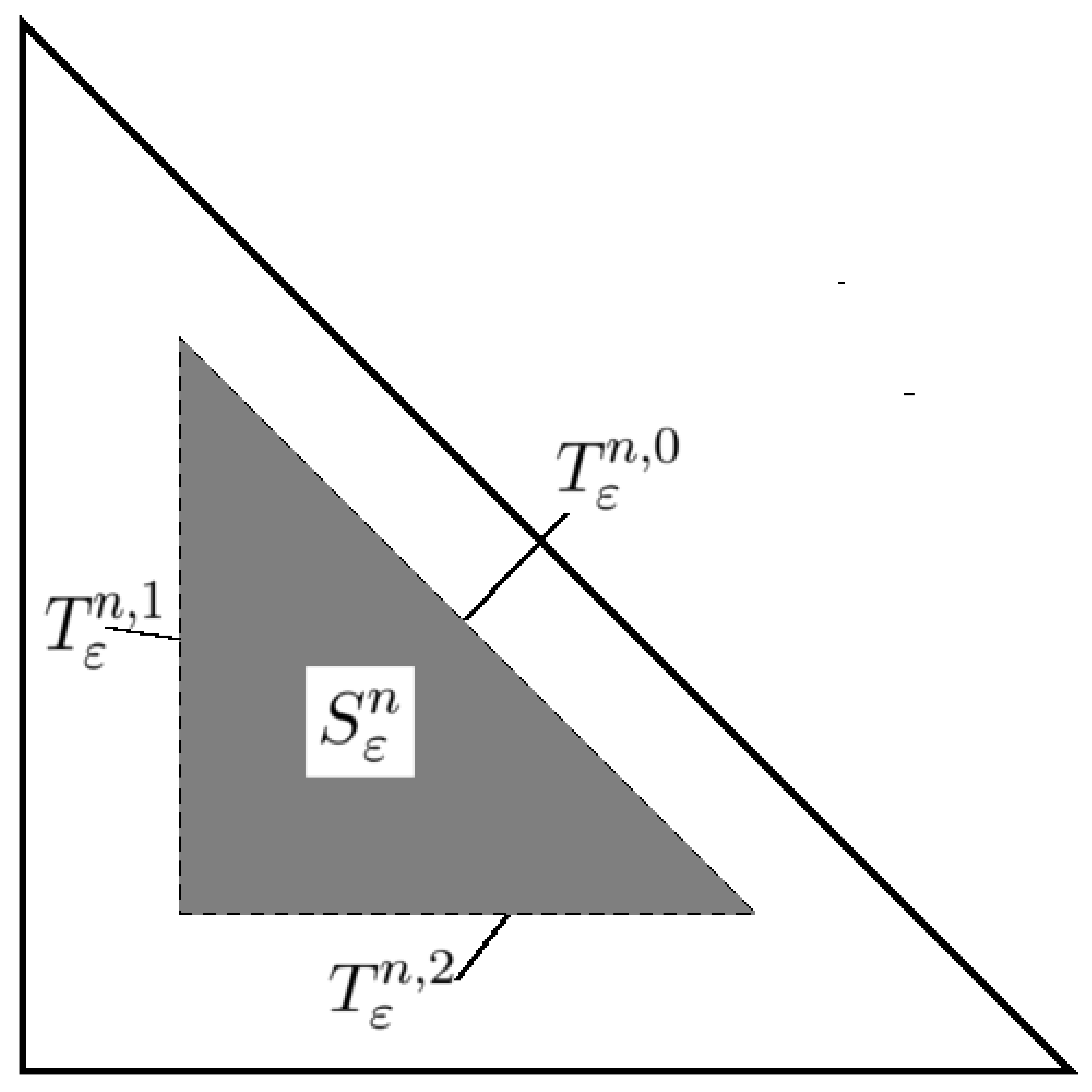}
\caption{Some notations used in the proof of Theorem \ref{ThSimplex}.}
\label{trianglefig}
\end{center}
\end{figure}
Also let $\nu_i$ be the Euclidean unit normal to $T^{n,i}_\epsilon$ (for any $\epsilon>0$). We note that $\partial S^n_\epsilon=\cup_{j=0}^n T^{n,i}_\epsilon.$

Following the first part of the proof of Theorem \ref{ThBall}, we show that $\Delta$ is a symmetric operator on the space $\mathscr C^\infty_b(S^n)$ which is dense (see Proposition \ref{propdensity}) in $H^1(S^n, g_{S^n}).$

To this aim we perform integration by parts twice. Let $u,v\in \mathscr C^\infty_b(S^n),$ then
\begin{align*}
&-\int_{S^n} v\Delta u d\Vol_{S^n}=-\int_{S^n} \Div(\sqrt{\det g_{S^n}} G_{S^n}^{-1}D u)v dx\\
=&\lim_{\epsilon\to 0^+} -\int_{S^n_\epsilon} \Div(\sqrt{\det g_{S^n}} G_{S^n}^{-1}D u)v dx\\
=&\lim_{\epsilon\to 0^+} \left(\int_{S^n_\epsilon}  D v^TG_{S^n}^{-1}D u \sqrt{\det g_{S^n}}dx-\sum_{i=0}^n\int_{T^{n,i}_\epsilon}  v\nu_i^T G_{S^n}^{-1}D u \sqrt{\det g_{S^n}}d\sigma\right)\\
=& \int_{S^n} \langle\grad u,\grad v\rangle_{g_{S^n}}  d\Vol_{S^n}-\sum_{i=0}^n\lim_{\epsilon\to 0^+}\int_{T^{n,i}_\epsilon}  \nu_i^T G_{S^n}^{-1}D u \sqrt{\det g_{S^n}}d\sigma\\
=&-\int_{S^n} v\Delta u d\Vol_{S^n}+\\
&\;\;\;\;\;\;\;\;\;\;\;\;\;\;\;\; \lim_{\epsilon\to 0^+}\int_{T^{n,i}_\epsilon} \left( u\nu_i^T G_{S^n}^{-1}D v-v\nu_i^T G_{S^n}^{-1}D u \right)\sqrt{\det g_{S^n}}d\sigma.
\end{align*}
Thus we need to prove that for any $u,v\in \mathscr C^\infty_b(S^n)$ and any $ i\in\{0,1,\dots, n\}$ we have
\begin{equation}
\lim_{\epsilon\to 0^+}\int_{T^{n,i}_\epsilon}u\nu_i^T G_{S^n}^{-1}D v\sqrt{\det g_{S^n}}d\sigma =0.\label{toshowsymmetry}
\end{equation}
For, it is sufficient to notice (using Lemma \ref{lemmasimplexinversemetric}) that for any $x\in T^{n,0}_\epsilon$
$$\nu_0^TG_{S^n}^{-1}\sqrt{\det g_{S^n}}=\sqrt{\frac \epsilon {\prod_{k=1}^nx_k}}(x_1,x_2,\dots,x_n)^T$$

and for any $x\in T^{n,i}_\epsilon$, $i=1,2,\dots,n$ \small
\begin{equation*}
\begin{split}
&\nu_i^TG_{S^n}^{-1}\sqrt{\det g_{S^n}}=\\
&\sqrt{\frac \epsilon {(1-\epsilon-\sum_{k=1,k\neq i}^nx_k)\prod_{k=1,k\neq i}^nx_k}}(x_1,x_2,\dots,x_{i-1},1-\epsilon,x_{i+1},\dots,x_n)^T.
\end{split}
\end{equation*}
\normalsize
Therefore we have
$$\left|\int_{T^{n,0}_\epsilon}u\nu_i^T G_{S^n}^{-1}D v\sqrt{\det g_{S^n}}d\sigma\right|\leq \sqrt\epsilon n\max_{S^n}(|Dv|_\infty |u|)\left\|\prod_{k=1}^n \sqrt{x_k}\right\|_{L^1(T^{n,0}_\epsilon)}\to 0$$
and, for any $i=1,2,\dots,n$
\begin{align*}
&\left|\int_{T^{n,i}_\epsilon}u\nu_i^T G_{S^n}^{-1}D v\sqrt{\det g_{S^n}}d\sigma\right|\\
\leq& \sqrt\epsilon n\max_{S^n}(|Dv|_\infty |u|)\left\| \left((1-\epsilon-\sum_{k=1,k\neq i}^nx_k)\prod_{k=1,k\neq i}^nx_k\right)^{-1/2}\right\|_{L^1(T^{n,i}_\epsilon)}\to 0
\end{align*}
and thus \eqref{toshowsymmetry} holds true. This shows that $\Delta$  is a symmetric operator on $\mathscr C^\infty_b(S^n),$ i.e., for any such $u$ and $v$
\small
\begin{equation}\label{simplexsymmetry}
\int_{S^n}u\Delta  v d\Vol_{S^n}=-\int_{S^n}\langle \grad u,\grad v\rangle_{g_{S^n}}d\Vol_{S^n} =\int_{S^n}v\Delta u d\Vol_{S^n}.
\end{equation}
\normalsize

Now we want to show that $\Delta$ has discrete spectrum $\sigma(\Delta_S)=\{\lambda_s:=s(s+\frac{n-1}{2}): s\in \N\}$ and the eigen-space of $\lambda_s$ is $\Span\{\psi_\alpha,|\alpha|=s\}$ (see Proposition \ref{orthogonalsystemsymplex}).

Instead of proving this directly, we rely on the known properties of the basis $\{\psi_\alpha\}$, namely \eqref{differentialproperty}, and we simply show that for smooth functions 
\begin{equation}\label{toshoweigen}
\Delta f=\mathcal D f,
\end{equation} 
this allows us to characterize $\sigma(\Delta)$ due to Theorem \ref{simplexeigenfunctions}. Then we apply Theorem \ref{ThOpTh} and the thesis follows.

We introduce the notation $h(x):=(1-\sum_{k=1}^nx_k)\prod_{k=1}^n x_k.$ It is worth to note that
$$\sqrt{h(x)}\partial_i \frac{x_i}{\sqrt{h(x)}}=\frac{1-\sum_{k\neq i}x_k}{2(1-\sum_{k=1}^nx_k)}=\frac 1 2\left(1+\frac{x_i}{1-\sum_{k=1 }^nx_k} \right) .$$
 For any smooth $f$ we have
\begin{align*}
&\Delta_{S^n}f\\
=&\sqrt{h(x)}\sum_{i=1}^n\partial_i\left(\frac {x_i}{\sqrt{h(x)}}(\partial_i f-\sum_{j=1}^nx_j\partial_j f)   \right)\\
=&\sum_{i=1}^n\left\{ \sqrt{h(x)}\partial_i \frac{x_i}{\sqrt{h(x)}}(\partial_i f-\sum_{j=1}^nx_j\partial_j f)+x_i\partial_i(\partial_i f-\sum_{j=1}^nx_j\partial_j f)\right\}\\
=&\sum_{i=1}^n\left\{\frac 1 2\left(1+\frac{x_i}{1-\sum_{k=1 }^nx_k} \right)(\partial_i f-\sum_{j=1}^nx_j\partial_j f)+x_i\partial_i(\partial_i f-\sum_{j=1}^nx_j\partial_j f)\right\}\\
=&-\frac 1 2\sum_{j=1}^nx_j\partial_j f\cdot\sum_{i=1}^n\left(1+\frac{x_i}{1-\sum_{k=1}^nx_k} \right)+\\
&\;\;\;\;\;\frac 1 2 \sum_{i=1}^n\partial_i f+\frac 1{2(1-\sum_{k=1}^nx_k)}\sum_{i=1}^nx_i\partial_i f\;\;+\\
&\;\;\;\;\;\sum_{i=1}^n\left\{x_i \partial_i^2 f-x_i\sum_{j\neq i}x_j\partial_{i,j}^2f -x_i^2 \partial_i^2 f-x_i \partial_i f \right\}\\
=&\sum_{i=1}^nx_i(1-x_i)\partial^2_i f-2\sum_{1\leq j<i\leq n}x_ix_j\partial_{i,j}^2 f+\frac 1 2 \sum_{i=1}^n\partial_i f\\
&\;\;+\;\;\left(\sum_{i=1}^nx_i\partial_i f\right)\cdot\left\{-\sum_{i=1}^n\left(\frac 1 2+\frac{x_i}{2\left(1-\sum_{k=1 }^nx_k\right)}\right)+\frac{1}{2\left(1-\sum_{k=1}^nx_k\right)}   -1\right\}\\
=&\sum_{i=1}^nx_i(1-x_i)\partial^2_i f-2\sum_{1\leq j<i\leq n}x_ix_j\partial_{i,j}^2 f+\frac 1 2 \sum_{i=1}^n\partial_i f\\
&\;\;+\;\;\left(  \sum_{i=1}^n x_i\partial_i f\right)\cdot\left\{-\frac{n+2}{2} +\frac{-\sum_{i=1}^nx_i+1 }{2(1-\sum_{k=1}^nx_k)} \right\}\\
=&\sum_{i=1}^nx_i(1-x_i)\partial^2_i f-2\sum_{1\leq j<i\leq n}x_ix_j\partial_{i,j}^2 f+\frac 1 2 \sum_{i=1}^n\partial_i f\\
&=\sum_{i=1}^nx_i \partial_{i,i}^2 f-2\sum_{1\leq i<j\leq n}x_i x_j \partial_{i,j}^2 f +\frac 1 2\sum_{i=1}^n(1-(n+1)x_i)\partial_i f\\
=& \mathcal D f.
\end{align*}
\end{proof}
\subsection{Proof of Proposition \ref{propnotalmostpolar}}
Let us first recall a result of Masamune \cite[Th. 3]{Ma99} which the proof of Proposition \ref{propnotalmostpolar} relies on. Assume $(M,g)$ to be a compact Riemaniann manifold and let $\Sigma$ be a submanifold of $M,$ let us define $\Delta_M$ as the standard Laplace Beltrami operator acting on $\mathscr C^\infty_c(M\setminus \Sigma).$ Then 
\begin{equation}\label{masamuneresult}
\Delta_M\text{ is essentially self-adjoint if and only if }\dim(M)-\dim(\Sigma)>3.
\end{equation}
\begin{proof}[Proof of Proposition \ref{propnotalmostpolar}]
Let $M:=S^n\subset \mathbb R^{n+1}$ and $\Sigma:=\{x\in M: x_{n+1}=0\}.$ Also introduce the notation $(x_1,x_2,\dots,x_n,x_{n+1})=(\xi,x_{n+1}).$

Let us assume by contradiction that $\mathscr C^\infty_c(B^n)$ is dense in $H^1(B^n,g_{B^n}).$ 
In view of the proof of Proposition \ref{propdensity} we have
\begin{equation}
\begin{split}
\mathscr H:=&\left(\mathscr C^\infty_c(B^n),\|\cdot\|_{1,2,g_{B^n}}\right){\leftrightarrows}_{\text{isometry}} \left(\mathscr C^\infty_{c,\text{even}}(M\setminus \Sigma),\|\cdot\|_{1,2,g_M}\right)=:\mathscr E_1.\\
&\left(\mathscr C^\infty_c(B^n),\|\cdot\|_{1,2,g_{B^n}}\right){\leftrightarrows}_{\text{isometry}} \left(\mathscr C^\infty_{c,\text{odd}}(M\setminus \Sigma),\|\cdot\|_{1,2,g_M}\right)=:\mathscr E_2.
\end{split}
\end{equation}
Here $\mathscr C^\infty_{c,\text{odd}}(M\setminus \Sigma)$ denotes the subspace 
$$\left\{u\in \mathscr C^\infty_{c}(M\setminus \Sigma),g_{\mathbb S^n}),\;u(\xi,x_{n+1})=-u(\xi,-x_{n+1})\, \forall (\xi,x_{n+1})\in M\setminus \Sigma\right\}$$
and $\mathscr C^\infty_{c,\text{even}}(M\setminus \Sigma)$ is defined similarly. Note that, given $u\in \mathscr C^\infty_c(M\setminus \Sigma)$ we can define $u_{\text{even}}:=1/2(u(\xi,x_{n+1})+u(\xi,-x_{n+1}))\in \mathscr E_1$ and $u_{\text{odd}}:=1/2(u(\xi,x_{n+1})-u(\xi,-x_{n+1}))\in \mathscr E_2$ such that $u=u_{\text{even}}+u_{\text{odd}}$.

The assumption $\mathscr C^\infty_c(B^n)$ dense in $H^1(B^n,g_{B^n})$ together with Theorem \ref{ThBall} and the isometry property of the map $E$ in the proof of Proposition \ref{propdensity} implies that the Laplace Beltrami operator $\Delta_1$ acting on $\mathscr E_1$ and $\Delta_2$ acting on $\mathscr E_2$ are essentially self-adjoint. Moreover, since
$\Delta_M u=\Delta_1 u_{\text{even}}+ \Delta_2 u_{\text{odd}}$ for any $u\in \mathscr C^\infty_c(M\setminus \Sigma),$ it follows that $\Delta_M$ itself is essentially self-adjoint.

On the other hand, $\dim \Sigma= n-1$ and $\dim M=n$, this is in contrast with Masamune's result \eqref{masamuneresult} and thus $\mathscr C^\infty_c(B^n)$ can not be dense in $H^1(B^n,g_{B^n})$ and thus $H^1(B^n,g_{B^n})\neq H^1_0(B^n,g_{B^n}).$ Note that, in view of \cite[Th. 1]{Ma05}, this is equivalent to the fact that $B^n$ is not a manifold with almost polar boundary.

The proof for the simplex can be done in a equivalent way but using the map $F$ defined in the proof of Proposition \ref{propdensity} instead of the map $E.$
\end{proof}

\appendix
\section{Pluripotential Theory on the complexified sphere and spherical harmonics }\label{appendix}
In this section we consider $\mathbb S^{n-1}$ as a compact subset of the complexified sphere $\mathcal S^{n-1}:=\{z\in \C^n:\sum z_i^2 =1\}.$  We can consider the space $\psh(\mathcal S^{n-1})$ of plurisubharmonic functions on the complex manifold $\mathcal S^{n-1}$ and form the usual upper envelope 
$$V_{\mathbb S^{n-1}}^*(z,\mathcal S^{n-1}):=\limsup_{\mathcal S^{n-1}\ni\z\to z}\,\sup\left\{u(\z):u\in \psh(\mathcal S^{n-1}), u|_{\mathbb S^{n-1}}\leq 0 \right\}$$ 
\emph{defining} the \emph{extremal plurisubharmonic function}; compare this definition with equation \eqref{efdef}. This is a locally bounded plurisubharmonic function which is maximal on $\mathcal S^{n-1}\setminus \mathbb S^{n-1};$ \cite{Ze91,Be82}.

On the other hand, it is clear that $\mathcal S^{n-1}$ is a irreducible algebraic sub-variety of $\C^n$ of pure dimension $n-1,$, hence we can use the result of Sadullaev \cite{Sa82} to get
$$V_{\mathbb S^{n-1}}^*(z,\mathcal S^{n-1})=\limsup_{\mathcal S^{n-1}\ni\z\to z}\,\sup\left\{\frac 1{\deg p}\log^+|p(\z)|:p\in \wp(\C^n), \|p\|_{\mathbb S^{n-1}}\leq 1 \right\}.$$ 
Here $\wp$ denotes the space of algebraic polynomials with complex coefficients. It is worth to stress that here $\deg$ denotes the degree of a polynomial on $\C^n$, not the degree over the coordinate ring of $\mathcal S^{n-1}.$

The operator $(\ddc)^{n-1}$ maps the space of locally bounded plurisubharmonic functions on positive Borel measures \cite{BeTa82,Be82} and 
\begin{equation}\label{eqmsph}
\mu_{\mathbb S^{n-1},\mathcal S^{n-1}}:=(\ddc V_{\mathbb S^{n-1}}^*(z,\mathcal S^{n-1}))^{n-1}
\end{equation}
is a probability measure, termed the pluripotential equilibrium measure of $\mathbb S^{n-1}.$

The invariance of the couple $\mathbb S^{n-1},\mathcal S^{n-1}$ and the operator $(\ddc)^{n-1}$ under  complex rotations can be used to show that $\mu_{\mathbb S^{n-1},\mathcal S^{n-1}}$ is indeed absolutely continuous with respect to the standard volume measure of $\mathbb S^{n-1}$ with constant density. Since $\mu_{\mathbb S^{n-1},\mathcal S^{n-1}}$ by definition has total mass $1$ its density with respect to the standard volume form is $c_{n-1}:=(2 \pi^{n/2}/\Gamma(n/2))^{-1}.$

In \cite[Prop. 4.1]{BoLeMaPi17} authors prove the formula
\begin{equation}\label{efsphere}
V_{\mathbb S^{n-1}}^*(z,\mathcal S^{n-1})=\frac 1 2 \log\left(|z|^2+\sqrt{|z|^4-1}  \right),\,\forall z\in \mathcal S^{n-1},
\end{equation}
we note that this function can be used to define the Baran metric on the sphere, due to the following differentiability property.
\begin{lemma}\label{differentiabilitylemma}
Let $x\in \mathbb S^{n-1},$ the function $V_{\mathbb S^{n-1}}(\cdot,\mathcal S^{n-1})$ has right tangent directional derivative at $x$ in any direction $i\cdot v$, for any $v\in T_{x}\mathbb S^{n-1},$ that is
$$\partial^+_{i\cdot v}V_{\mathbb S^{n-1}}(x,\mathcal S^{n-1}):=\frac d {dt}V_{\mathbb S^{n-1}}(\gamma(t),\mathcal S^{n-1})|_{t=0}\in \R,$$
where $\gamma:[0,1]\mapsto \mathcal S^{n-1}$ is any differentiable arc with $\gamma(0)=x$, $\gamma'(0^+)=i\cdot v.$

Moreover we have $\partial_{i\cdot v}V_{\mathbb S^{n-1}}(x,\mathcal S^{n-1})=|v|.$ 
\end{lemma} 
\begin{proof}
The problem is clearly rotation independent. We can thus assume $x=(1,0,\dots,0)=e_1$ and $v=|v|(0,1,0,\dots,0)=|v| e_2$ without loss of generality.

Let us introduce the arc
$$z(t):= \sqrt{1+|v|^2\log^2 (1+t)}e_1+|v|\log(1+t) e_2, t\in[0,+\infty[.$$
It is easy to verify that $z$ enjoys the properties
\begin{align*}
&z(t)\in \mathcal S^{n-1}, \forall t\in [0,+\infty[,\\
&z(0)=x,\\
&\frac d {dt}z(0^+)=i\cdot v. 
\end{align*}
Thus we are left to show that, setting $u(t):= V_{\mathbb S^{n-1}}^*(z(t),\mathcal S^{n-1}),$ we have
$\frac d {dt} u(0^+)=|v|.$

Let us note first that $|z(t)|^2=1+2|v|^2 \log^2(1+t)$, then we can compute
\begin{align*}
u(t)=&\frac 1 2 \log\left[1+2|v|^2 \log^2(1+t)+\sqrt{4|v|^2 \log^2(1+t)(1+|v|^2 \log^2(1+t))} \right]
\\
=&\frac 1 2 \log\left[1+2|v|^2 \log^2(1+t)+2|v|\log(1+t)\sqrt{1+|v|^2 \log^2(1+t)} \right]\\
\sim&\frac 1 2 \log\left[1+2|v|^2 t^2+2|v| t\sqrt{1+|v|^2 t^2}\right]\\
\sim&\frac 1 2 \log(1+ 2|v|t)\sim |v| t,\;\;\text{ as } t\to 0^+.
\end{align*}
Therefore 
$$u'(0^+)=\lim_{t\to 0^+}\frac{u(t)-u(0)}{t}=\lim_{t\to 0^+}\frac{u(t)}{t}=|v|.$$
\end{proof}

Due to Lemma \ref{differentiabilitylemma} we can define the \emph{Baran metric} on the real unit sphere by setting
$$\delta_{\mathbb S^{n-1}}(x,v):= \partial_{i\cdot v}V_{\mathbb S^{n-1}}^*(x,\mathcal S^{n-1})=|v|,$$
note the analogy of the partial derivative taken in the Lemma with the definition of the Baran metric in the standard "flat" case. 

Using the Parallelogram Identity we can define for any $x\in \mathbb S^{n-1}$ and any $u,v\in T_x\mathbb S^{n-1}$ the scalar product related to the Baran metric as
\begin{equation*}
\begin{split}
\langle u,v\rangle_{g_{\mathbb S^{n-1}}(x)}:=&\frac{\delta_{\mathbb S^{n-1}}^2(x,u+v)-\delta_{\mathbb S^{n-1}}^2(x,u-v)}{4}\\
=&\frac{|u+v|^2-|u-v|^2}4=\langle u, v\rangle_{\R^n},
\end{split}
\end{equation*}
that turns out to coincide with the standard (round) metric. 

It is very well known that the Laplace Beltrami operator on the real unit sphere (endowed with the round metric) has a discrete diverging set eigenvalues and its eigenfunctions are polynomials: the \emph{spherical harmonics}.

These observations lead automatically to the desired conclusion that we state as a corollary.
\begin{cor}\label{corollarysphere}
The eigenfunctions of the Laplace Beltrami operator with respect to the Baran metric on the real unit sphere are the orthogonal polynomials with respect to the pluripotential equilibrium measure $\mu_{\mathbb S^{n-1},\mathcal S^{n-1}}$ of the real unit sphere $\mathbb S^{n-1}$ in the complexified sphere $\mathcal S^{n-1}.$
\end{cor}
\bibliographystyle{abbrv}
\bibliography{biblio}

\def\lasp{\leavevmode\raise.45ex\hbox{$\lhook$}}
\begin{thebibliography}{10}
\providecommand{\url}[1]{\normalfont{#1}}
\providecommand{\urlprefix}{Available from: }

\bibitem{SaTo97}
Saff~EB, Totik~V. Logarithmic potentials with external fields. Springer-Verlag
  Berlin; 1997.

\bibitem{Sa10}
Saff~EB. Logarithmic potential theory with applications to approximation
  theory. Surveys in Approximation Theory. 2010;\hspace{0pt}5:165--200.

\bibitem{StaTo92}
Stahl~H, Totik~V. General orthogonal polynomials. Vol.~43 of Encyclopedia of
  Mathematics and its Applications. Cambridge University Press, Cambridge;
  1992.

\bibitem{WAL}
Walsh~JL. Interpolation and approximation by rational function on complex
  domains. AMS; 1929.

\bibitem{Rans}
Ransford~T. Potential theory in the complex plane. Vol.~28 of London
  Mathematical Society Student Texts. Cambridge University Press, Cambridge;
  1995; \urlprefix\url{http://dx.doi.org/10.1017/CBO9780511623776}.

\bibitem{Kli}
Klimek~M. Pluripotential theory. Oxford Univ. Press; 1991.

\bibitem{Klo05}
Ko{\l}odziej~S. The complex {M}onge-{A}mp\`ere equation and pluripotential
  theory. Mem Amer Math Soc. 2005;\hspace{0pt}178(840):x+64;
  \urlprefix\url{http://dx.doi.org/10.1090/memo/0840}.

\bibitem{BeTa82}
Bedford~E, Taylor~BA. A new capacity for plurisubharmonic functions. Acta
  Mathematica. 1982;\hspace{0pt}149(1):1--40.

\bibitem{LevSurRocky92}
Bloom~T, Bos~L, Christensen~C, Levenberg~N. Polynomial interpolation of
  holomorphic functions in {${\bf C}$} and {${\bf C}^n$}. Rocky Mountain J
  Math. 1992;\hspace{0pt}22(2):441--470;
  \urlprefix\url{http://dx.doi.org/10.1216/rmjm/1181072740}.

\bibitem{LevSur}
Bloom~T, Bos~LP, Calvi~J, Levenberg~N. Approximation in $\mathbb{C}^n$. Ann
  Polon Math. 2012;\hspace{0pt}(106):53--81.

\bibitem{BaBi14}
Baran~M, Bialas-Ciez~L. H\"older continuity of the {G}reen function and
  {M}arkov brothers' inequality. Constr Approx.
  2014;\hspace{0pt}40(1):121--140;
  \urlprefix\url{http://dx.doi.org/10.1007/s00365-013-9224-0}.

\bibitem{Ba92}
Baran~M. Plurisubharmonic extremal functions and complex foliations for the
  complement of convex sets in {${\bf R}^n$}. Michigan Math J.
  1992;\hspace{0pt}39(3):395--404;
  \urlprefix\url{http://dx.doi.org/10.1307/mmj/1029004594}.

\bibitem{Si81}
Siciak~J. Extremal plurisubharmonic functions in $\mathbb{C}^n$. Ann Polon
  Math. 1981;\hspace{0pt}319:175--211.

\bibitem{Ze91}
Zeriahi~A. Fonction de {G}reen pluriclomplexe \`a pole \`a l'infini sur un
  espace de {S}tein parabolique et application. Mathematica Scandinavica.
  1991;\hspace{0pt}69:89--126.

\bibitem{Za74I}
Zaharjuta~VP. Extremal plurisubharmonic functions, {H}ilbert scales, and the
  isomorphism of spaces of analytic functions of several variables. i,
  (russian). Teor Funkci\u{i} Funkcional Anal i Prilo\v{z}en.
  1974;\hspace{0pt}127(19):133--157.

\bibitem{Za74II}
Zaharjuta~VP. Extremal plurisubharmonic functions, hilbert scales, and the
  isomorphism of spaces of analytic functions of several variables. ii,
  (russian). Teor Funkci\u{i} Funkcional Anal i Prilo\v{z}en.
  1974;\hspace{0pt}127(21):65--83.

\bibitem{Za75}
Zaharjuta~VP. Transfinite diameter, {C}hebyshev constant and capacity for
  compacta in $\mathbb{C}^n$. Math USSR Sb. 1975;\hspace{0pt}25(350).

\bibitem{BeBo10}
Berman~R, Boucksom~S. Growth of balls of holomorphic sections and energy at
  equilibrium. Invent Math. 2010;\hspace{0pt}181(2):337--394;
  \urlprefix\url{http://dx.doi.org/10.1007/s00222-010-0248-9}.

\bibitem{BeBoNy11}
Berman~R, Boucksom~S, Witt~Nystr{\"o}m~D. Fekete points and convergence towards
  equilibrium measures on complex manifolds. Acta Math.
  2011;\hspace{0pt}207(1):1--27;
  \urlprefix\url{http://dx.doi.org/10.1007/s11511-011-0067-x}.

\bibitem{B97}
Bloom~T. Orthogonal polynomials in {$\mathbb C^n$}. Indiana Univ Math J.
  1997;\hspace{0pt}46(2):427--452.

\bibitem{ZeZe10}
Zeitouni~O, Zelditch~S. Large deviations of empirical measures of zeros of
  random polynomials. Int Math Res Not IMRN. 2010;\hspace{0pt}(20):3935--3992;
  \urlprefix\url{http://dx.doi.org/10.1093/imrn/rnp233}.

\bibitem{BlLe15}
Bloom~T, Levenberg~N. Random polynomials and pluripotential-theoretic extremal
  functions. Potential Anal. 2015;\hspace{0pt}42(2):311--334;
  \urlprefix\url{http://dx.doi.org/10.1007/s11118-014-9435-4}.

\bibitem{PrYeeAa15}
Pritsker~IE, Yeager~AM. Zeros of polynomials with random coefficients. J Approx
  Theory. 2015;\hspace{0pt}189:88--100;
  \urlprefix\url{http://dx.doi.org/10.1016/j.jat.2014.09.003}.

\bibitem{MaOr10}
Marzo~J, Ortega-Cerd{\`a}~J. Equidistribution of {F}ekete points on the sphere.
  Constr Approx. 2010;\hspace{0pt}32(3):513--521;
  \urlprefix\url{http://dx.doi.org/10.1007/s00365-009-9051-5}.

\bibitem{BeOr15}
Berman~RJ, Ortega-Cerd\`a~J. Sampling of real multivariate polynomials and
  pluripotential theory. arxiv preprint 150900956. 2015;\hspace{0pt}.

\bibitem{MaOr15}
Marzo~J, Ortega-Cerd\`a~J. Uniformly bounded orthonormal polynomials on the
  sphere. Bull Lond Math Soc. 2015;\hspace{0pt}47(5):883--891;
  \urlprefix\url{http://dx.doi.org/10.1112/blms/bdv061}.

\bibitem{BlBoLe12}
Bloom~T, Bos~L, Levenberg~N. The transfinite diameter of the real ball and
  simplex. Ann Polon Math. 2012;\hspace{0pt}106:83--96;
  \urlprefix\url{http://dx.doi.org/10.4064/ap106-0-6}.

\bibitem{BoLeVi12}
Bos~L, Levenberg~N, Vianello~M. Convex bodies and optimal admissible meshes.
  2012.

\bibitem{Ba95}
Baran~Ma. Complex equilibrium measure and {B}ernstein type theorems for compact
  sets in {${\bf R}^n$}. Proc Amer Math Soc. 1995;\hspace{0pt}123(2):485--494;
  \urlprefix\url{http://dx.doi.org/10.2307/2160906}.

\bibitem{Sa82}
Sadullaev~A. An estimates for polynomials on analytic sets. Math URSS
  Izvestiya. 1982;\hspace{0pt}20(3):493--502.

\bibitem{DuXu01}
Dunkl~CF, Xu~Y. Orthogonal polynomials of several variables. Vol.~81 of
  Encyclopedia of Mathematics and its Applications. Cambridge University Press,
  Cambridge; 2001; \urlprefix\url{http://dx.doi.org/10.1017/CBO9780511565717}.

\bibitem{DuXu14}
Dunkl~CF, Xu~Y. Orthogonal polynomials of several variables. 2nd ed.
  Encyclopedia of Mathematics and its Applications; Cambridge University Press,
  Cambridge; 2014; \urlprefix\url{http://dx.doi.org/10.1017/CBO9781107786134}.

\bibitem{BoLeWa04}
Bos~L, Levenberg~N, Waldron~S. Metrics associated to multivariate polynomial
  inequalities. In: Advances in constructive approximation: {V}anderbilt 2003.
  Mod. Methods Math.; Nashboro Press, Brentwood, TN; 2004. p. 133--147.

\bibitem{BoLeWa08}
Bos~L, Levenberg~N, Waldron~S. Pseudometrics, distances and multivariate
  polynomial inequalities. J Approx Theory. 2008;\hspace{0pt}153(1):80--96;
  \urlprefix\url{http://dx.doi.org/10.1016/j.jat.2008.02.002}.

\bibitem{Ze12}
Zelditch~S. Pluri-potential theory on {G}rauert tubes of real analytic
  {R}iemannian manifolds, {I}. In: Spectral geometry. Vol.~84 of Proc. Sympos.
  Pure Math.; Amer. Math. Soc., Providence, RI; 2012. p. 299--339;
  \urlprefix\url{http://dx.doi.org/10.1090/pspum/084/1363}.

\bibitem{Ze14}
Zelditch~S. Ergodicity and intersections of nodal sets and geodesics on real
  analytic surfaces. J Differential Geom. 2014;\hspace{0pt}96(2):305--351;
  \urlprefix\url{http://projecteuclid.org/euclid.jdg/1393424920}.

\bibitem{Ze07}
Zelditch~S. Complex zeros of real ergodic eigenfunctions. Invent Math.
  2007;\hspace{0pt}167(2):419--443;
  \urlprefix\url{http://dx.doi.org/10.1007/s00222-006-0024-z}.

\bibitem{Ba92b}
Baran~Ma. Bernstein type theorems for compact sets in {${\bf R}^n$}. J Approx
  Theory. 1992;\hspace{0pt}69(2):156--166;
  \urlprefix\url{http://dx.doi.org/10.1016/0021-9045(92)90139-F}.

\bibitem{CL08}
Calvi~JP, Levenberg~N. Uniform approximation by discrete least squares
  polynomials. J Approx Theory. 2008;\hspace{0pt}152(1):82--100.

\bibitem{Kr11}
Kro{\'o}~A. On optimal polynomial meshes. J Approx Theory.
  2011;\hspace{0pt}163(9):1107--1124.

\bibitem{BoDeSoVi10}
Bos~LP, Marchi~SD, Sommariva~A, Vianello~M. Weakly admissible meshes and
  discrete extremal sets. Numer Math Theory Methods Appl.
  2011;\hspace{0pt}41(1):1--12.

\bibitem{DePiSoVi15}
Marchi~SD, Piazzon~F, Sommariva~A, Vianello~M. Polynomial meshes: Computation
  and approximation. Proceedings of CMMSE. 2015;\hspace{0pt}:414--425.

\bibitem{Pi17}
Piazzon~F. Optimal polynomial admissible meshes on some classes of compact
  subsets of {$\Bbb{R}^d$}. J Approx Theory. 2016;\hspace{0pt}207:241--264;
  \urlprefix\url{http://dx.doi.org/10.1016/j.jat.2016.02.015}.

\bibitem{He99}
Hebey~E. Nonlinear analysis on manifolds: {S}obolev spaces and inequalities.
  Vol.~5 of Courant Lecture Notes in Mathematics. New York University, Courant
  Institute of Mathematical Sciences, New York; American Mathematical Society,
  Providence, RI; 1999.

\bibitem{Ma05}
Masamune~J. Analysis of the {L}aplacian of an incomplete manifold with almost
  polar boundary. Rend Mat Appl (7). 2005;\hspace{0pt}25(1):109--126.

\bibitem{Ma99}
Masamune~J. Essential self-adjointness of {L}aplacians on {R}iemannian
  manifolds with fractal boundary. Comm Partial Differential Equations.
  1999;\hspace{0pt}24(3-4):749--757;
  \urlprefix\url{http://dx.doi.org/10.1080/03605309908821442}.

\bibitem{Da95}
Davies~EB. Spectral theory and differential operators. Vol.~42 of Cambridge
  Studies in Advanced Mathematics. Cambridge University Press, Cambridge; 1995;
  \urlprefix\url{http://dx.doi.org/10.1017/CBO9780511623721}.

\bibitem{Be82}
Bedford~E. The operator $(dd^c)^n$ on complex spaces. Lecture Notes in Math
  Springer, Berlin-New York. 1982;\hspace{0pt}919:294--323.

\bibitem{BoLeMaPi17}
Bos~L, Levenberg~N, Ma{\lasp}u~S, Piazzon~F. A weighted extremal function and
  equilibrium measure. Math Scand to appear (https://arxivorg/abs/150507749).
  2017;\hspace{0pt}.

\end{thebibliography}
\end{document}